\documentclass[11pt,usenames,dvipsnames]{article}
\usepackage[T1]{fontenc}
\usepackage{lmodern}
\usepackage{dsfont}
\usepackage{multicol}
\usepackage{enumerate}
\usepackage[mathscr]{euscript}
\usepackage{mathrsfs}
\usepackage{mathtools}
\usepackage{authblk}
\usepackage{amssymb,amsmath,amsthm,epsfig,textcomp,amsthm}
\usepackage[colorlinks=true,breaklinks=true,linkcolor=MidnightBlue!100,citecolor=Maroon!100]{hyperref}
\usepackage{graphics,psfrag,graphicx,color}
\usepackage[graphicx]{realboxes}
\numberwithin{equation}{section}
\usepackage[normalem]{ulem}

\usepackage{enumitem}  
\setlist{  
  listparindent=\parindent,
  parsep=0pt,
}

\usepackage[x11names,table]{xcolor}
\usepackage{booktabs}
\usepackage{rotating}
\usepackage{varwidth}
\usepackage{comment}
\usepackage{float}
\floatstyle{boxed} 

\usepackage{tabularx,ragged2e,booktabs,caption}
\newcolumntype{C}[1]{>{\Centering}m{#1}}

\usepackage{multirow}

\usepackage{algorithm}
\usepackage{algpseudocode}
\usepackage{etoolbox}
\makeatletter
\newcommand*{\algrule}[1][\algorithmicindent]{%
  \makebox[#1][l]{%
    \hspace*{.2em}
    \vrule height .75\baselineskip depth .25\baselineskip
  }
}

\newcount\ALG@printindent@tempcnta
\def\ALG@printindent{%
    \ifnum \theALG@nested>0
    \ifx\ALG@text\ALG@x@notext
    \else
    \unskip
    \ALG@printindent@tempcnta=1
    \loop
    \algrule[\csname ALG@ind@\the\ALG@printindent@tempcnta\endcsname]%
    \advance \ALG@printindent@tempcnta 1
    \ifnum \ALG@printindent@tempcnta<\numexpr\theALG@nested+1\relax
    \repeat
    \fi
    \fi
}
\patchcmd{\ALG@doentity}{\noindent\hskip\ALG@tlm}{\ALG@printindent}{}{\errmessage{failed to patch}}
\patchcmd{\ALG@doentity}{\item[]\nointerlineskip}{}{}{} 

\newcommand{\pdual}[1]{\left\langle#1\right\rangle}
\newcommand{\prom}[1]{ \{\!\!\{ {#1} \}\!\!\} }
\newcommand{\jump}[1]{ [ \! [ {#1} ] \! ] }
\newcommand{\inner}[2]{ (\!( {#1} , {#2} )\!) }
\newcommand{\innerh}[2]{ (\!( {#1} , {#2} )\!)_h }
\newcommand{\normh}[1]{ \triple{#1}_h }

\newcommand{\kap}{\boldsymbol{\kappa}}

\newcommand{\vertiii}[1]{{\left\vert\kern-0.25ex\left\vert\kern-0.25ex\left\vert #1 
    \right\vert\kern-0.25ex\right\vert\kern-0.25ex\right\vert}}
\newcommand{\J}{J}    

\newcommand{\bdiv}{\textbf{div}}
\newcommand{\bn}{\boldsymbol{n}}
\newcommand{\bq}{\boldsymbol{q}}
\newcommand{\bt}{\boldsymbol{t}}

\newcommand{\bv}{\boldsymbol{v}}

\newcommand{\bx}{\boldsymbol{x}}
\newcommand{\by}{\boldsymbol{y}}

\newcommand{\bH}{\boldsymbol{H}}
\newcommand{\bI}{\boldsymbol{I}}

\newcommand{\bL}{\boldsymbol{L}}

\newcommand{\bV}{\boldsymbol{V}}

\newcommand{\normalh}{{\boldsymbol{\nu}_h}}
\newcommand{\bsy}[1]{\boldsymbol{#1}}
\newcommand{\mbf}[1]{\boldsymbol{#1}}

\newcommand{\compD}{\Omega_h}
\newcommand{\erroru}{\varepsilon^u}

\newcommand{\errorq}{\boldsymbol{\varepsilon}^{\boldsymbol{q}}}

\newcommand{\triple}[1]{|\!|\!|{#1}|\!|\!|}

\newcommand{\SDtN}{S_1}
\newcommand{\SNtD}{S_2}
\newcommand{\Sh}{S_h}
\newcommand{\Sn}{S_n}
\newcommand{\T}{T}
\newcommand{\Tomega}{T_\omega}

\newcommand{\mc}[1]{\mathcal{#1}}

\newcommand{\md}[1]{\mathds{#1}}

\definecolor{dgreen}{rgb}{0.0, 0.5, 0.0}
\newtheorem{thm}{Theorem}
\newtheorem{lem}{Lemma}
\newtheorem{crl}{Corollary}

\title{Afternote to \textit{Coupling at a distance}:\\convergence analysis and \textit{a priori} error estimates   \\ {\small\textit{Dedicated to the memory of Francisco--Javier Sayas.}}}

\author[1,2,4]{Nestor S\'anchez}
\author[3]{Tonatiuh S\'anchez-Vizuet}
\author[1,2]{Manuel E. Solano}
\affil[1]{{\small Departamento de Ingenier\'ia Matem\'atica, Universidad de Concepci\'on, Concepci\'on, Chile.}}
\affil[2]{{\small Centro de Investigaci\'on 
en Ingenier\'ia Matem\'atica (CI$^2$MA), Universidad de Concepci\'on, Concepci\'on, Chile.}}
\affil[3]{{\small Department of Mathematics, The University of Arizona, USA.}}
\affil[4]{{\small Instituto de Matem\'aticas, Unidad Juriquilla. Universidad Nacional Aut\'onoma de M\'exico.}}

\date{}

\begin{document}

\maketitle
\begin{abstract}
In their article \textit{``Coupling at a distance HDG and BEM''} \cite{CoSaSo2012}, Cockburn, Sayas and Solano proposed an iterative coupling of the hybridizable discontinuous Galerkin method (HDG) and the boundary element method (BEM) to solve an exterior Dirichlet problem. The novelty of the numerical scheme consisted of using a computational domain for the HDG discretization  whose boundary did not coincide with the coupling interface. In their article, the authors provided extensive numerical evidence for convergence, but the proof of convergence and the error analysis remained elusive at that time. In this article we fill the gap by proving the convergence of a relaxation of the algorithm and providing \textit{a priori} error estimates for the numerical solution.
\end{abstract}

{\bf Key words}: Hybridizable discontinuous Galerkin (HDG), boundary element Method (BEM), coupling HDG-BEM, curved interface,  transfer path method.

\noindent
{\bf Mathematics Subject Classifications (2020)}: 65N15, 65N30, 65R20.
\section{Introduction}
The goal of this article is to conclude the work started by Cockburn, Sayas and Solano in the article \textit{Coupling at a distance} \cite{CoSaSo2012}, where an iterative solution method for a classic exterior elliptic problem was introduced. The proposed scheme amounted to a Schur complement-style algorithm that alternates between a Hybridizable Discontinuous Galerkin Method (HDG) for an interior problem and the Boundary Element Method (BEM) for an exterior problem. At the time of publication, the novelty of the method resided in the use of \textit{non-touching} grids for the discretization of each of the two problems. The ready availability of two separate, \textit{uncoupled}, codes for each of the discretization methods and the eagerness to show the viability of such a non-touching coupling led to the choice of an iterative alternating procedure---even though the problem in question is in fact linear.

When \cite{CoSaSo2012} was published, the technique for transferring information between the two grids had only been recently incorporated into the HDG literature \cite{CoSo2012} and, despite the fact that convincing numerical evidence of convergence at an optimal rate was provided, a rigorous analysis of the coupled scheme proved elusive at the time. A few years after \textit{Coupling at a distance} appeared, a method for the analysis of HDG discretizations involving the transfer technique---that we now like to call the \textit{transfer path method}---was developed in \cite{CoQiuSo2014} for interior elliptic problems. Since then, both the transfer technique and the analysis method have been successfully employed for the study of linear \cite{OySoZu2020,SoVa2019,SoVa2022}, and non-linear \cite{OySoSu2021,SaSaSo2021,SaSaSo:2021a,SaSo2019,SaCeSo2020} interior problems, as well as problems with interfaces \cite{QiSoVe2016,SoTeNgPe2021}, however the analysis of the HDG-BEM coupling had fallen by the wayside and remained unfinished.

The current special issue honoring Francisco--Javier Sayas, one of the co-authors of the original article, seemed like the perfect venue for the missing analysis. In that sense, the present communication shall not be considered a novel contribution, but rather the conclusion, long overdue, of the original work, an after-note to the original work \textit{Coupling at a distance}. With that in mind, we will stick to the iterative alternating procedure proposed in \cite{CoSaSo2012}, even if a more efficient monolithic approach where the HDG and BEM discrete systems---along with the discrete coupling terms---are solved simultaneously is possible. The study of such a monolithic scheme applied to nonlinear problems is the subject of ongoing work that will be communicated in a separate publication \cite{SaSaSo:2021b}.

{The method proposed in \cite{CoSaSo2012}, rather than approaching the problem as a single coupled unit, follows the spirit of domain decomposition methods. It relies on an iterative approximation of a Dirichlet to Neumann mapping through the independent solution of an interior and an exterior problem that communicate through their Dirichlet and Neumann traces. Since these two problems are dealt with independent solvers, we will analyze their discretizations separately. After establishing the well posedness of the independent discretizations, we will then prove that, at the discrete level, the alternating solution of an interior Dirichlet and (with HDG) an exterior Neumann problem (with BEM) converges to the solution of the original unbounded problem. This latter result constitutes the main contribution of this article.}

We will describe the problem setting and its reformulation as a system of coupled interior/exterior problems at the continuous level in Section \ref{sec:ContinuousFormulation}. The discretizations of the interior problem and the boundary integral formulation for the exterior problem are described respectively in sections \ref{sec:HDG} and \ref{sec:BEM}. Finally, in Section \ref{sec:Coupling}, we show that it is possible to define a relaxation of the iterative process presented in \cite{CoSaSo2012}, alternating between the solution of the interior and the boundary problems, that converges to the solution of the original problem.

\section{Continuous Formulation}\label{sec:ContinuousFormulation}
\subsection{Problem setting}
Consider a bounded domain $\Omega_0 \subset {\mathds{R}^2}$ that has a smooth parametrizable boundary that will be denoted by $\Gamma_0:= \partial \Omega_0$. We will denote the unbounded complement of its closure by $\Omega_0^c =: \md{R}^d \setminus \overline{\Omega_0}$. In this chapter, we will be concerned with the analysis of a discretization for the following diffusion problem 
	\begin{subequations}\label{eq:ext-diff-problem}
	\begin{align}
	&& && \nabla \cdot \bq^{\text{tot}} &= f & & \text{ in }  \Omega_0^c , && 
	\label{eq:ext-diff-problem_1} \\
	&& && \bq^{\text{tot}} + \boldsymbol{\kappa}\, \nabla u^{\text{tot}} &= 0 &  &\text{ in } \Omega_0^c , && 
	\label{eq:ext-diff-problem_2} \\
	&& && u^{\text{tot}} &= u_0 & & \text{ on } \Gamma_0, && 
	\label{eq:ext-diff-problem_3}\\
	&& && {u^{\text{tot}}} &= \mathcal O(1) & & \text{ as } \boldsymbol x \to \infty. && \label{eq:ext-diff-problem_4}
	\end{align}
	\end{subequations}
The function $f$ will be taken to be compactly supported and  square integrable on $\Omega^c_0$. The diffusion coefficient $\kap$ is a strictly positive matrix-valued function such that, denoting the identity matrix is as $\mathbf I$, the difference $(\mathbf I-\boldsymbol{\kappa})$ is compactly supported in $\Omega_0^c$. This condition implies that outside of $\text{supp}(\mathbf I-\boldsymbol{\kappa})$ equations \eqref{eq:ext-diff-problem_1} and \eqref{eq:ext-diff-problem_2} in fact coincide with Poisson's equation. We will also require that there exist positive constants $\underline{\kap}$ and $\overline{\kap}$ such that, for any component function $\kappa_{ij}$ of $\boldsymbol\kappa$ it holds that
    \begin{equation*}
    \underline{\kap} \leq \kappa_{ij}(\bx) \leq\overline{\kap} \qquad \forall \, \bx \in \Omega.
    \end{equation*}	
The Dirichlet boundary data $u_0$ will be considered to be an element of the trace space $H^{1/2}(\Gamma_0)$. The radiation condition at infinity \eqref{eq:ext-diff-problem_4} is equivalent to assuming that there is a constant $u_{\infty}$ such that $u=u_{\infty} + \mc{O}(|\bx|^{-1})$   \cite{McLean2002}.

\textcolor{red}{}

\begin{figure}[tb]
\centering 
\includegraphics[width=0.4\linewidth]{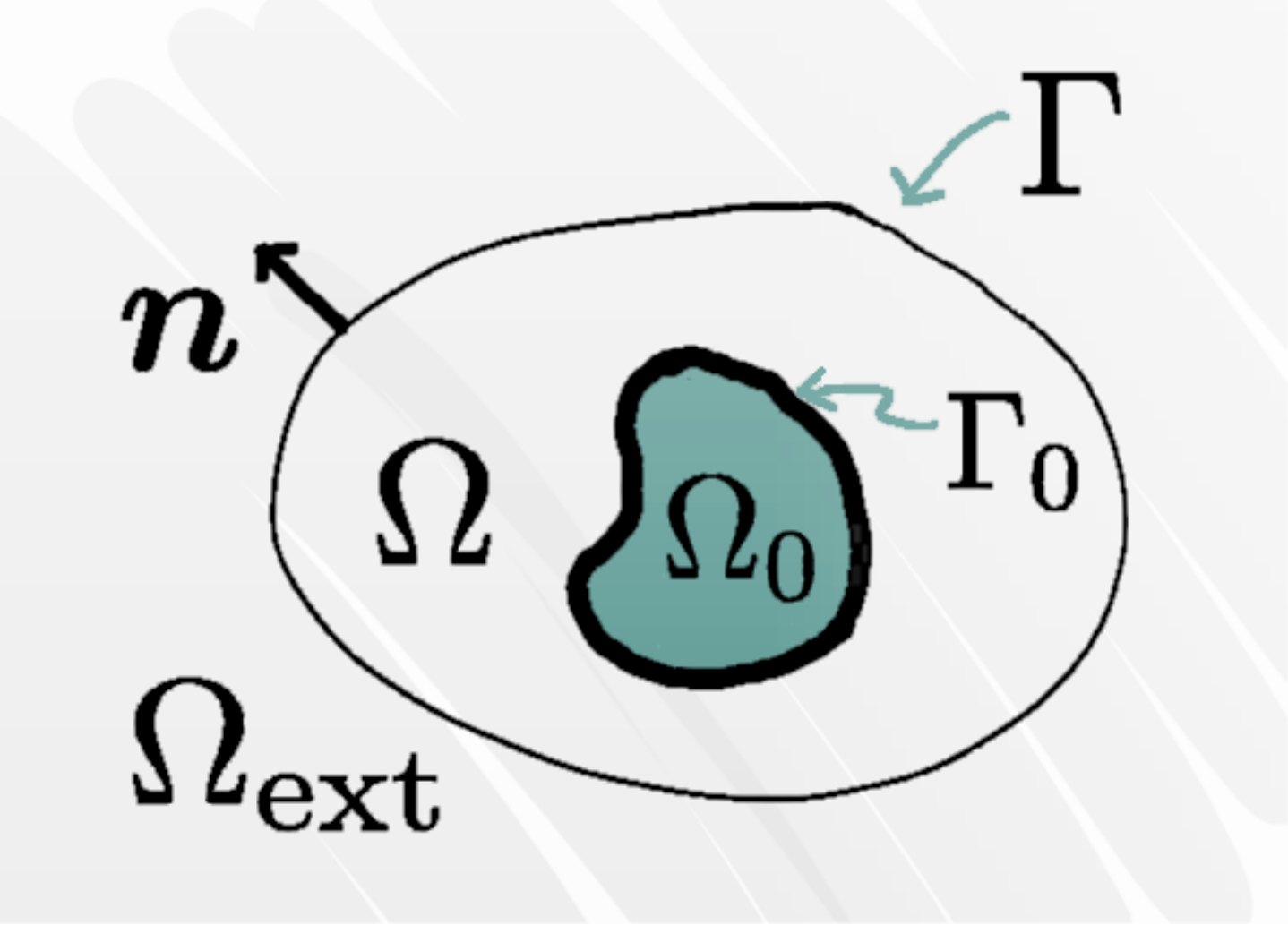} \qquad \qquad
\includegraphics[width=0.4\linewidth]{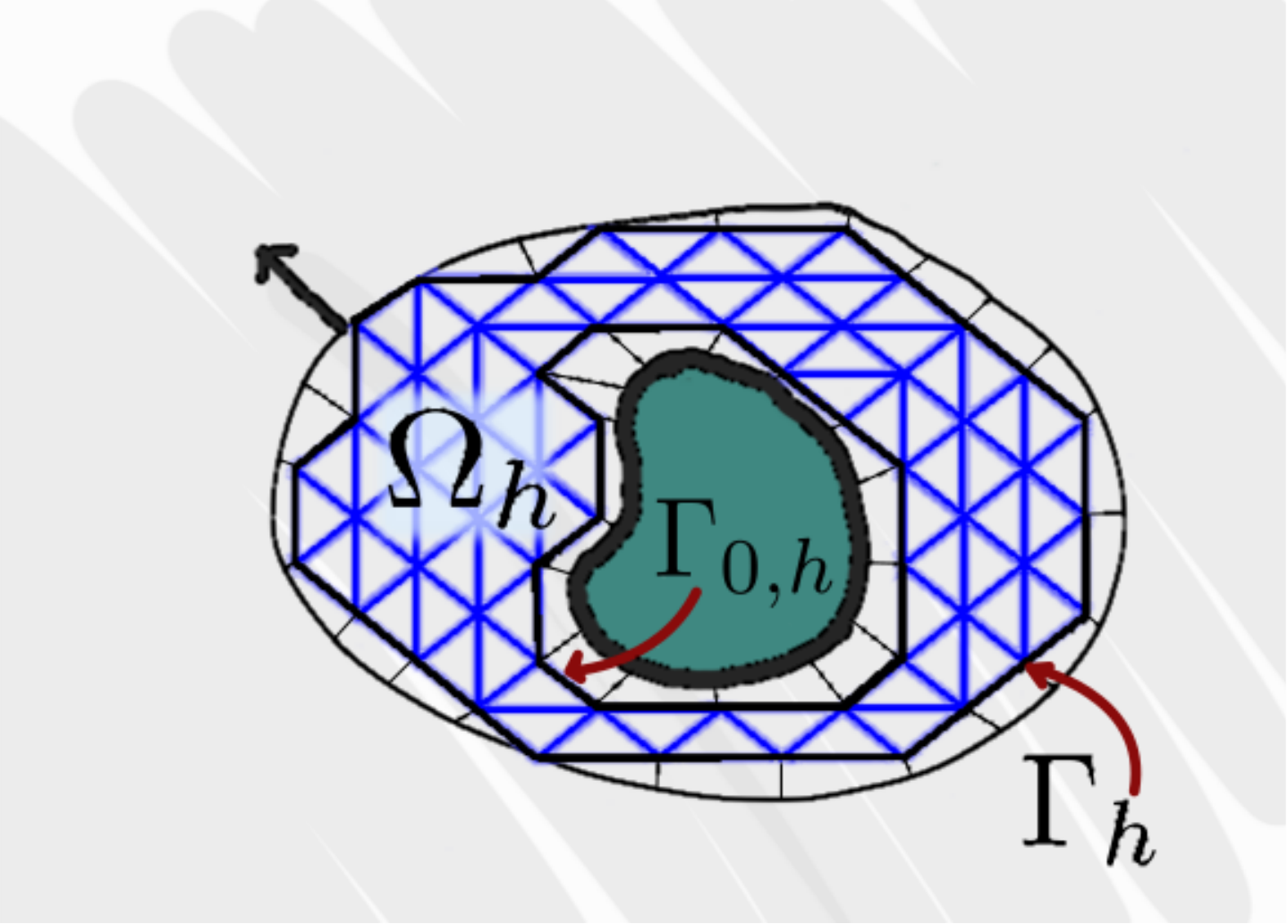}
\caption{Left: The artificial boundary $\Gamma$ splits the domain of definition of Problem \eqref{eq:ext-diff-problem} into an unbounded region $\Omega_{\text{ext}}$ and a bounded annular domain $\Omega$. Right: The computational domain $\Omega_h$ is discretized by an un-fitted triangulation (blue), with boundary $\Gamma_h\cup\Gamma_{0,h}$. }\label{fig:Geometry}
\end{figure}
%

\subsection{Interior and exterior problems}

To deal with the unboundedness of the domain, later on we will make use of an integral representation that will reduce the computations to a bounded domain. To this avail, we introduce an artificial, smoothly parametrizable interface $\Gamma$ enclosing $\Omega_0$, the support of $f$ and the support of $(\mathbf I - \boldsymbol\kappa)$. We will also require that $\Gamma\cap\Gamma_0=\varnothing$. The domain interior to $\Gamma$ will be denoted $\Omega$, while the unbounded complementary region will be denoted $\Omega_{\text{ext}}$. The boundary of $\Omega$ will be denoted as $\partial\Omega$ and consists of two disjoint components: the artificial boundary $\Gamma$ and the original problem boundary $\Gamma_0$, so that $\partial \Omega = \Gamma \cup \Gamma_0$.
We will denote the unit normal vector to $\partial \Omega$, pointing in the direction of $\Omega_{\text{ext}}$ for points in $\Gamma$ and in the direction of $\Omega_0$ for points in $\Gamma_0,$ by $\boldsymbol n$. This geometric decomposition, depicted in Figure \ref{fig:Geometry}, splits our region of interest into two disjoint domains and allows us to rewrite the problem \eqref{eq:ext-diff-problem}  in terms of an interior and an exterior problem coupled by continuity conditions at the artificial boundary $\Gamma$.

Since we aim to use an integral equation formulation, for the exterior problem we will prefer a second order formulation and will eliminate $\boldsymbol q^{\text{ext}}$ from the system. We will represent the solutions to \eqref{eq:ext-diff-problem} as the superposition
\begin{equation*}
u^{\text{tot}} = u + u^{\text{ext}} \;\;\text{ and }\;\; \boldsymbol q^{\text{tot}} = \boldsymbol q +  \nabla u^{\text{ext}},
\end{equation*}
where the functions $u$ and $\boldsymbol q$ are supported in $\Omega$, while $u^{\text{ext}}$ is supported in $\Omega_{\text{ext}}$.
The pair $(u,\boldsymbol q)$ satisfies the interior problem
	\begin{subequations}\label{eq:interior-problem}
	\begin{align}
	&& && \nabla \cdot \bq &= f & & \text{ in }  \Omega, && \\
	&& && \bq + \boldsymbol{\kappa}\, \nabla u &= 0 &  &\text{ in } \Omega ,&& \\
	\label{eq:ContinuityU}
	&& && u &= g & & \text{ on } \Gamma, && \\
	\label{eq:ContinuityQN}
	&& && \boldsymbol q\cdot\boldsymbol n & = \lambda & & \text{ on } \Gamma, && \\
	&& && u &= u_0 & & \text{ on } \Gamma_0. &&
	\end{align}
	\end{subequations}
On the other hand, the exterior function $u^{\text{ext}}$ satisfies
	\begin{subequations}\label{eq:exterior_problem}
	\begin{align}
	&& && -\Delta
	\label{eq:external}
	u^{\text{ext}} &= 0 & & \text{ in }  \Omega_{\text{ext}}, && \\
	\label{eq:ContinuityUext}
	&& && u^{\text{ext}} &= g & & \text{ on } \Gamma, && \\
	\label{eq:ContinuityQNext}
	&& && \nabla u^{\text{ext}}\cdot\boldsymbol n & = -\lambda & & \text{ on } \Gamma, && \\
	&& &&{u^{\text{tot}}} & = \mathcal O(1)  & & \text{as $|\boldsymbol x|\to\infty$ .}&&
	\end{align}
	\end{subequations}

Above, the boundary value $g\in H^{1/2}(\Gamma)$ corresponds to the trace of $u^{\text{tot}}$ over the artificial boundary $\Gamma$, while $\lambda\in H^{-1/2}(\Gamma)$ is the value of the normal flux. These two functions are unknown at this point and will have to be retrieved as part the solution process. However, the knowledge of $g$ (resp.  $\lambda$) is enough to fully determine the solution to \eqref{eq:interior-problem} or \eqref{eq:exterior_problem} considered as independent problems---as long as the equation containing $\lambda$ (resp. $g$) is removed from the system. This observation will motivate the alternating solution scheme to be described in Section \ref{sec:Coupling}.

\subsection{Boundary integral formulation for the exterior problem.}
We will now reformulate \eqref{eq:exterior_problem} as a boundary integral equation. To do that, we will make use of some standard results from potential theory; we refer the reader interested in further details to the classic references \cite{HsWe:2004,McLean2002} for a comprehensive  account, or to \cite{HsStWe2017} for a more concise treatment.

We start by introducing the \textit{single layer} and \textit{double layer} potentials defined respectively for $\eta\in H^{1/2}(\Gamma)$, $\mu\in H^{-1/2}(\Gamma)$ and $\boldsymbol x\in \mathbb R^2\setminus\Gamma$ as
\begin{alignat*}{6}
\mathcal S\mu(\boldsymbol x) :=\,& \int_\Gamma G(\boldsymbol x,\boldsymbol y)\mu(\boldsymbol y)\,d\Gamma_{\boldsymbol y} \qquad&& \text{(Single layer)}, \\
\mathcal D\eta(\boldsymbol x) :=\,& \int_\Gamma \partial_{\boldsymbol n(\boldsymbol y)}G(\boldsymbol x,\boldsymbol y)\eta(\boldsymbol y)\,d\Gamma_{\boldsymbol y} \qquad&& \text{(Double layer)},
\end{alignat*}
where $G(\boldsymbol x,\boldsymbol y)$ is the Green function for Poisson's equation. The functions defined by these two potentials satisfy Equation \eqref{eq:external}, and the following \textit{jump conditions}
\[
\jump{\mathcal S\mu} : = 0\,,  \qquad \jump{\nabla\mathcal S\mu} : = \mu\,, \qquad 
\jump{\mathcal D\eta} : = -\eta\,,  \qquad \jump{\nabla\mathcal D\eta} : = 0\,, 
\]
where the \textit{jump} operator is defined for $\boldsymbol y\in\Gamma$ and scalar and vector functions $v$ and $\boldsymbol v$ respectively as
\begin{equation}\label{eq:JumpOperator}
\jump{v} : = \lim_{\epsilon\to 0}\left(v(\boldsymbol y - \epsilon\boldsymbol n)-v(\boldsymbol y + \epsilon\boldsymbol n)\right) \quad \text{and} \quad
\jump{\boldsymbol v} : = \lim_{\epsilon\to 0}\left(\boldsymbol v(\boldsymbol y - \epsilon\boldsymbol n)-\boldsymbol v(\boldsymbol y + \epsilon\boldsymbol n)\right)\cdot\boldsymbol n(\boldsymbol y).
\end{equation}
In a similar fashion we can define the average operators as
\begin{equation}\label{eq:jumpAndave}
\prom{v} : = \frac{1}{2}\lim_{\epsilon\to 0}\left(v(\boldsymbol y - \epsilon\boldsymbol n)+v(\boldsymbol y + \epsilon\boldsymbol n)\right) \quad \text{and} \quad
\prom{\boldsymbol v} : = \frac{1}{2}\lim_{\epsilon\to 0}\left(\boldsymbol v(\boldsymbol y - \epsilon\boldsymbol n)+\boldsymbol v(\boldsymbol y + \epsilon\boldsymbol n)\right)\cdot\boldsymbol n(\boldsymbol y),
\end{equation}
and use them to define the following boundary integral operators
\[
\mathcal V \mu := \prom{\mathcal S \mu}\,, \qquad \mathcal K^\prime \mu := \prom{\nabla \left(\mathcal S \mu\right)}\,, \qquad
\mathcal K \eta := \prom{\mathcal D \eta}\,, \quad\text{and}\quad \mathcal W \eta := -\prom{\nabla \left(\mathcal D \eta\right)}.
\]
We are now in a position to recast the exterior problem \eqref{eq:exterior_problem} in terms of boundary integral equations. To that avail, we will represent $u^{\text{ext}}$ in $\Omega_{\text{ext}}$ as
\begin{equation}\label{eq:IntegralRepresentation}
u^{\text{ext}} = 
\mathcal D g  - \mathcal S\lambda + u_\infty 
\end{equation}
and extend it by zero for $\boldsymbol x\in \Omega$. The constant $u_\infty$ captures the far field behavior of the function and will have to be determined. Since $u^{\text{ext}}\equiv0$ in $\Omega$, by applying the integral operators above to the integral representation \eqref{eq:IntegralRepresentation},  the boundary condition \eqref{eq:ContinuityUext} leads to
\[
\prom{u^{\text{ext}}} = \tfrac{1}{2} g = \mathcal K g - \mathcal V\lambda + \tfrac{1}{2}u_\infty,
\]
giving rise to the integral equation
\begin{subequations}\label{eq:IntegralProblem}
\begin{equation}\label{integral-identity}
\left(\tfrac{1}{2} - \mathcal K\right)g = -\mathcal V\lambda + \tfrac{1}{2}u_\infty.
\end{equation}
To ensure that $u^{\text{ext}}=u_\infty$ as $|\boldsymbol x|\to\infty$, we must impose the additional restriction
\begin{equation}\label{eq:MeanConstraint}
\int_\Gamma\lambda = 0.
\end{equation}
\end{subequations}
Equation \eqref{integral-identity} will be used as part of the alternating scheme described in Section \ref{sec:Coupling}, where an approximation of $\lambda$ will be produced by a numerical solution of the interior problem \eqref{eq:interior-problem} and the density $g$ solving \eqref{integral-identity} will be then used as the Dirichlet datum for \eqref{eq:interior-problem}.

Therefore,  if $\Gamma$ has two continuous derivatives and $\lambda\in H^{-1/2}(\Gamma)$ is problem data satisfying the constraint \eqref{eq:MeanConstraint}, then the unique solvability of equation \eqref{eq:IntegralProblem} and continuous dependence on problem data follow from standard results in boundary integral equations (see, for instance \cite[Section 6.4]{Kress2014}). Moreover, there exists a constant $c>0$, depending only on $\Gamma$ and the norms of $(1/2-\mathcal K)^{-1}$ and $\mathcal V$, such that
\begin{align}\label{ineq:ContDependenceBIE}
\|g\|_{1/2,\Gamma} \leq c\|\lambda\|_{-1/2,\Gamma}.
\end{align}
Moreover, from this estimate and the representation formula \eqref{eq:IntegralRepresentation}, it follows that there exists $C_{\text{BIE}}>0$ such that
\begin{align}\label{ineq:ContDependenceBIE2}
\|u^{\text{ext}}\|_{\Omega} \leq C_{\text{BIE}}\|\lambda\|_{-1/2,\Gamma}+|u_\infty|.
\end{align}
\color{black}
\subsection{Variational formulation for the interior problem}\label{sec:VariationalFormulation}

In this Section, we will study the interior Dirichlet boundary value problem obtained from \eqref{eq:interior-problem} by removing \eqref{eq:ContinuityQN} altogether and considering that the boundary trace $g$, appearing in \eqref{eq:ContinuityU}, is known. This yields the problem
	\begin{subequations}\label{eq:interior-problem no coupled}
	\begin{align}
	\label{eq:interior-problem no coupled A}
	&& && \nabla \cdot \bq &= f & & \text{ in }  \Omega, && \\
	\label{eq:interior-problem no coupled B}
	&& && \boldsymbol{\kappa}^{-1}\, \bq + \nabla u &= 0 &  &\text{ in } \Omega ,&& \\
	\label{eq:interior-problem no coupled C}
	&& && u &= \xi_0 & & \text{ on } \partial \Omega.
	\end{align}
	\end{subequations}
Above, the source term $f\in L^2(\Omega)$ and the Dirichlet boundary data $\xi_0\in H^{1/2}(\partial\Omega)$ is given by
\[
\xi_0 = \left\{\begin{array}{cl}
            u_0 & \text{on } \Gamma_0, \\
            g   & \text{on } \Gamma.
\end{array}\right.
\]
To derive the weak formulation of this system, we  test \eqref{eq:interior-problem no coupled A} with an arbitrary $w\in L^2(\Omega)$ and \eqref{eq:interior-problem no coupled B} with $\bv \in \bH(\bdiv;\Omega)$, integrate by parts and incorporate \eqref{eq:interior-problem no coupled C} leading to
    \begin{align*}
    (\nabla \cdot \bq , w )_\Omega =\,&\phantom{-} (f,w)_\Omega \\
    (\boldsymbol{\kappa}^{-1}\, \bq, \bv)_\Omega - ( u, \nabla \cdot \bv)_\Omega =\,&- \langle \bv \cdot \bn , \xi_0 \rangle_{\partial \Omega},
    \end{align*}
{where $( \cdot, \cdot )_\Omega$ and $\langle \cdot, \cdot \rangle_{\partial \Omega}$ denote the $L^2$-inner products over $\Omega$ and $\partial \Omega$, respectively.}
From the three preceding equations, we arrive at the variational problem: 

Find $(\bq,u)\in \bH(\bdiv;\Omega) \times L^2(\Omega)$ such that
    \begin{subequations}\label{scheme-augmented-continuous}
    \begin{align}
     \widetilde{\mc{A}}(\bq, \bv) + \widetilde{\mc{B}}(\bv,u) &= \widetilde{\mc{F}}_1(\bv)  \qquad \forall \, \bv\in \bH(\bdiv;\Omega), \\
     \widetilde{\mc{B}}(\bq,w) &= \widetilde{\mc{F}}_2(w) \qquad \forall \, w\in L^2(\Omega),
    \end{align}
    \end{subequations}
where the bilinear forms $\widetilde{\mc{A}}: \bH(\bdiv;\Omega) \times \bH(\bdiv;\Omega) \to \md{R}$, \ $\widetilde{\mc{B}}: \bH(\bdiv;\Omega) \times L^2(\Omega) \to \md{R}$, and the functionals $\widetilde{\mc{F}_1 }: \bH(\bdiv;\Omega) \to \md{R}$ and $\widetilde{\mc{F}_2} : L^2(\Omega) \to \md{R}$ are defined by
    \begin{subequations}
    \begin{align*}
   \widetilde{\mc{A}}(\bq,\bv) &:=   (\kap^{-1} \bq, \bv)_{\Omega}, \\
    \widetilde{\mc{B}}(\bq, w) &:= -(w, \nabla \cdot \bq)_{\Omega}, \\
    \widetilde{\mc{F}}_1(\bv) &:=  - \langle \xi_0, \bv \cdot {\bn} \rangle_{\partial\Omega}, \\
    \widetilde{\mc{F}}_2( w) &:=  -(f, w )_{\Omega}.
    \end{align*}
    \end{subequations}
 
The well-posedness of \eqref{scheme-augmented-continuous} follows from standard arguments of Bab\v{u}ska-Brezzi theory \cite[Sec. 2.4]{Gatica:2014} and the solution satisfies
\begin{equation}\label{ineq:ContDependence}
\|\boldsymbol{q}\|_{{\rm div},\Omega} + \|u\|_{0,\Omega}\ \leq C_{\text{stab}}\overline{\boldsymbol \kappa}^{1/2}\left( \|f\|_{0,\Omega} + \|\xi_0\|_{1/2,\partial\Omega} \right) = C_{\text{stab}}\overline{\boldsymbol \kappa}^{1/2}\left( \|f\|_{0,\Omega} + \|g\|_{1/2,\Gamma} +    \|u_0\|_{1/2,\Gamma_0}\right).
\end{equation}

We will, however, not solve the problem as stated above and instead will consider a slightly different version posed in a subdomain. {This approach, known as the \textit{transfer path method} will be described in detail in Section \ref{sec:transfer}, and will require us first to discuss the geometric setting of the discretization, which we will do next.}
\section{HDG discretization of the interior problem}\label{sec:HDG}
\subsection{Geometric setting and notation}
\paragraph{The computational domain.}
We will consider, a family of polygonal subdomains $\Omega_h\subset\Omega$ that approximate $\Omega$ in the sense that the Lebesgue measure $\mu(\Omega\setminus\Omega_h)\to0$, as $h\to0$. We will refer to any such $\Omega_h$ as a \textit{computational domain} and will triangulate $\overline{\Omega}_h$ by a  shape-regular triangulation $\mc{T}_h$ as depicted in Figure \ref{fig:Geometry}. A generic element in $\mathcal T_h$ will be denoted by $T$ and the mesh parameter $h$ will be defined as diameter of a circle inscribing an element $T\in\mathcal T_h$. The set $\partial \mc{T}_h := \bigcup \{ \partial T: T\in \mc{T}_h\}$, will be referred to as \textit{the skeleton} of the triangulation. The set of edges, $e$, of $\mc{T}_h$ will be denoted by $\mc{E}_h$ and we will distinguish between those edges lying entirely in the computational boundary 
\[
\mathcal E^\partial_h :=\left\{e\in\mathcal E_h : e\cap\partial\Omega_h = e\right\},
\]
and those that are either interior or have at most their endpoints in the computational boundary
\[
\mathcal E^\circ_h :=\left\{e\in\mathcal E_h : e\cap\partial\Omega_h \neq e\right\}.
\]
We will refer to the former as \textit{boundary edges} and to the latter as \textit{interior edges}. Note that $\mc{E}_h = \mc{E}^\partial_h\cup \mc{E}^\circ_h$. 

Just as the boundary associated to the continuous problem \eqref{eq:interior-problem} has two separate connected components, the boundary of the computational domain can be split as $\partial\Omega_h = \Gamma_h \cup \Gamma_{h,0}$, where
\[
\Gamma_h :=\left\{e\in \mathcal T_h : d(e,\Gamma)\leq d(e,\Gamma_0)\right\} \quad \text{ and } \quad \Gamma_{h,0} :=\left\{e\in \mathcal T_h : d(e,\Gamma_0)< d(e,\Gamma)\right\}. 
\]

 We will require that the computational domain $\Omega_h$ and the triangulation $\mathcal T_h$ satisfy the following \textit{local proximity condition}: for any point in the computational boundary $\partial\Omega_h$, the minimum distance between $\boldsymbol x$ and the boundary $\partial\Omega=\Gamma\cup\Gamma_{0}$ should be, at most, of the same order of magnitude as the diameter of the smallest triangle $T\in\mathcal T_h$, such that $\boldsymbol x \in T$. In view of this condition, the process of mesh refinement should not be understood as a sequence of finer triangulations for a fixed computational domain $\Omega_h$. Instead, as the mesh diameter $h\to0$, the process involves the passage through a sequence of pairs domain/triangulation $(\Omega_h,\mathcal T_h)$ that satisfy the local proximity condition and exhaust the original domain $\Omega$ as the refinement progresses. We refer the reader to \cite{SaSaSo:2021a}, where this condition is discoursed in more detail, and to \cite{SaCeSo2020} where an algorithm for building a sequence $\{(\Omega_h,\mathcal T_h)\}_h$ is described.

\paragraph{Mesh-dependent subspaces and inner products.}
For the discrete formulation we will have introduce the following mesh-dependent inner products
	\begin{alignat*}{6}
	(u, w)_{\mc{T}_h} &:= \sum_{T\in \mc{T}_h} \int_T u \, w  \qquad &&\forall \, u, w \in L^2(\mc{T}_h),  \\
	(\bq, \bv)_{\mc{T}_h} &:= \sum_{T\in \mc{T}_h} \int_T \bq\cdot\bv  \qquad &&\forall \, \bq, \bv \in \bL^2(\mc{T}_h),  \\
	\langle u, w\rangle_{\partial \mc{T}_h} &:= \sum_{T\in \mc{T}_h} \int_{\partial T} u \, w  \qquad &&\forall \, u, w \in L^2(\partial \mc{T}_h),  \\	
	\langle u, w\rangle_{\partial \mc{T}_h \setminus \Gamma_h} &:= \sum_{T\in \mc{T}_h} \sum_{e\in \partial T \setminus \Gamma_h}\int_{e} u \, w  \qquad &&\forall \, u, w \in L^2(\partial \mc{T}_h), \\ 
	\langle u, w\rangle_{\partial \mc{T}_h \setminus \Gamma_{h,0}} &:= \sum_{T\in \mc{T}_h} \sum_{e\in \partial T \setminus \Gamma_{h,0}}\int_{e} u \, w  \qquad &&\forall \, u, w \in L^2(\partial \mc{T}_h).
	\end{alignat*}	
These inner products induce mesh-dependent norms that will be denoted, respectively, by
    \begin{equation*}
    \| w \|_{\compD} := (w,w)_{\mc{T}_h}^{1/2}, \qquad \| w \|_{\partial \mc{T}_h} := \langle w,w\rangle_{\partial \mc{T}_h}^{1/2} \quad  \text{ and } \quad   \| w \|_{\Gamma_h} := \langle w, w \rangle_{\partial \mc{T}_h \setminus \Gamma_h}^{1/2}.
    \end{equation*}
The finite dimensional discontinuous polynomial subspaces that will be used for discretization, for {$k\geq 0$}, are given by 
    \begin{subequations}
	\begin{align*}
	\bV_h &:= \{\bv\in \bL^2(\mc{T}_h) : \bv|_T \in [\md{P}_k(T)]^2, \ \forall \ T \in \mc{T}_h \}, \\
	W_h &:= \{w\in L^2(\mc{T}_h) : w|_T \in \md{P}_k(T), \ \forall \ T \in \mc{T}_h \}, \\
	M_h &:= \{\mu\in L^2(\mc{E}_h) : \mu|_T \in \md{P}_k(F), \ \forall \ F \in \mc{E}_h \},
	\end{align*}
	\end{subequations}
where, $\md{P}_k(T)$ denotes the space of polynomials of degree at most $k$ defined in $T\in \mc{T}_h$. Similarly, $\md{P}_k(e)$ denotes the space of polynomials of degree at most $k$ defined over a face $e\in \mc{E}_h$. 

\paragraph{Extension patches and extrapolation.}
Since the discrete spaces are defined only over the elements of the triangulation we will need to define a way to compute our approximations in the region $\Omega\setminus\Omega_h$ between the boundary and the computational boundary. To this purpose, we will tesselate this region as follows. Let:
\begin{itemize}
\item $\boldsymbol x_1$ and $\boldsymbol x_2$ be the endpoints of a boundary edge $e\in\partial\Omega_h$.
\item $\overline{\boldsymbol x}_1$ and $\overline{\boldsymbol x}_2$ be the corresponding points in $\partial\Omega$---as determined by the mapping \eqref{eq:BoundaryMapping}.
\item $\boldsymbol \sigma_1$ and $\boldsymbol \sigma_2$ the straight segments connecting $\overline{\boldsymbol x}_1$ to $\boldsymbol x_1$ and $\overline{\boldsymbol x}_2$ to $\boldsymbol x_2$. 
\end{itemize}

We will refer to the open region of $\Omega\setminus\Omega_h$ delimited by $e$, $\boldsymbol \sigma_1$ and $\boldsymbol \sigma_2$ and the segment of $\partial\Omega$ connecting $\overline{\boldsymbol x}_1$ to $\overline{\boldsymbol x}_2$ as an \textit{extension patch} and will denote it by $T_e^{\text{ext}}$. It is clear that for every $e\in\Gamma_h$ there is one and only one such $T_e^{\text{ext}}$ (this justifies subindex in the notation) and that $\overline{\Omega\setminus\Omega_h} = \cup_{e\in\Gamma_h}\overline{T}_e^{\text{ext}}$.

It also follows from this construction that for every $T_e^{\text{ext}}$ there is only one element $T_e$ in the triangulation such that $\overline{K}_e^{\text{ext}}\cap \overline{T}_e= e$. We will use this fact to define an extrapolation operator that will extend the value of the piecewise polynomial functions defined on $T_e$ onto the corresponding extension patch $T_e^{\text{ext}}$, thus extending functions the discrete spaces above into the full domain $\Omega$. With this in mind, we will define the values of polynomial function $p$ on $T_e^{\text{ext}}$ by extrapolating the values of the corresponding polynomial from $T_e$, and will denote its as $Ep(\boldsymbol x)$ for any $\boldsymbol x \in T_e^{\text{ext}}$. 

For a given domain $\Omega_h$ and corresponding triangulation  $\mathcal T_h$, the usual notion of the exterior normal vector is well defined for almost all points in the boundary, with the possible exception of the vertices of the triangulation. We will  define the exterior normal vector to the computational domain, $\boldsymbol n_h$ in the usual manner, and extend the definition to $\boldsymbol n_h(\boldsymbol x)=\boldsymbol\sigma(\boldsymbol x)$  for those vertices for which the standard normal vector is not well defined. {On the other hand, we will define the unit normal vector exterior \textit{to each element $T\in\mathcal T_h$} as $\normalh$, which will coincide with the exterior normal $\boldsymbol n_h$ on element edges belonging to the computational boundary $\Gamma_h$.}

Finally,  for every edge $e\in\mathcal E_h^\partial$ we will denote the ratio between its distance to the boundary and the diameter, $h_{T_e}$, of its parent element  as $r_e : = d(e,\partial\Omega)/h_{T_e}$, and will define the \textit{boundary proximity parameter} as
\begin{equation*}
R_h:= \max_{e\in\mathcal E_h^\partial}\, r_e,
\end{equation*}
and will assume for this work that the family of admissible domains and triangulations $(\Omega_h,\mathcal T_h)$ is such that: 1) $R_h\to0$ as $h\to0$, and 2) $\|\boldsymbol n_h-\boldsymbol n\|_\infty= o(h^{1/2})$ as $h\to0$, where the normal $\boldsymbol n_h$ should be understood as coinciding with $\boldsymbol\sigma$ for those points in which the standard normal vector is not defined.

\subsection{Transferal of boundary conditions}\label{sec:transfer}

Having introduced all the necessary geometric concepts we can now return to the interior problem \eqref{eq:interior-problem no coupled} which we will now pose in a polygonal computational domain $\Omega_h\subset\Omega$ satisfying the admissibility requirements discussed in the previous section. In addition, we will need to define a bijective\footnote{As numerous numerical experiments have shown \cite{CoSo2012,CoSo2014,SaSo2019,SaCeSo2020}, the algorithm is robust with respect to the particular choice for this mapping, so long as distance between $\boldsymbol x$ and its corresponding $\overline{\boldsymbol x}$ remains comparable to the local mesh diameter. In this article we will limit ourselves to consider solely those computational domains $\Omega_h$ for which such a mapping exists.} mapping
\begin{align}\label{eq:BoundaryMapping}
\phi: \partial\Omega_h \;& \longrightarrow  \partial\Omega\\
\nonumber
\boldsymbol x \;&\longmapsto \overline{\boldsymbol x}
\end{align}
assigning a point $\overline{ \boldsymbol x} \in \partial\Omega$ to every point $\boldsymbol x \in \partial\Omega_h$. 

For any fixed computational domain $\Omega_h$, the solution pair to \eqref{scheme-augmented-continuous} satisfies the related problem
	\begin{subequations}\label{eq:interior-problem no coupled in Dh}
	\begin{align}
	&& && \nabla \cdot \bq &= f \label{eq:interior1}& & \text{ in }  \Omega_h, && \\
	&& && \boldsymbol{\kappa}^{-1}\, \bq + \nabla u &= 0 &  &\text{ in } \Omega_h ,\label{eq:interior2}&& \\
	&& && u &=  \varphi^{\bq}_0 & & \text{ on }   \partial \Omega_h\label{eq:interior3},
	\end{align}
	\end{subequations}
where the boundary condition {$\varphi^{\bq}_0$} can be calculated by integrating equation \eqref{eq:interior-problem no coupled B} along a path connecting $\partial \Omega$ to $\partial \Omega_h$. More precisely, if we denote the distance between $\boldsymbol x$ and $\overline{\boldsymbol x}$ by $l(\boldsymbol x)$, and by $\boldsymbol t$ the unit vector $(\overline{\boldsymbol x}-\boldsymbol x)/|\overline{\boldsymbol x}-\boldsymbol x|$, the boundary conditions on $\Gamma_h$  can be expressed in terms of the flux $\boldsymbol q$ and the trace of $u$ on $\partial \Omega$, as
    \begin{equation}\label{def:xi_o-tilde}
     \varphi^{\bq}_0(\bx) := \xi_0\circ\phi(\bx) + \int_0^{l(\bx)} \boldsymbol{\kappa}^{-1} \bq (\bx + \bt(\bx) s) \cdot \bt(\bx) ds \qquad \forall \, \bx\in \partial \Omega_h.
    \end{equation}
Note that the required bijectivity of {$\phi(\boldsymbol x)$} implies that $\boldsymbol t$ can not be tangent { to a boundary edge}. Thus, the solution of \eqref{scheme-augmented-continuous} also satisfies the abstract formulation
    \begin{subequations}
    \begin{align*}
    \mc{A}(\bq, \bv) + \mc{A}_{T}(\bq, \bv) +  \mc{B}(\bv,u) &= \mc{F}_{1}(\bv)  \qquad \forall \, \bv\in \bH(\bdiv;\Omega_h), \\
     \mc{B}(\bq,w) &= \mc{F}_{2}(w) \qquad \forall \, w\in L^2(\Omega_h),
    \end{align*}
    \end{subequations}
where the bilinear forms $\mc{A}: \bH(\bdiv;\Omega_h) \times \bH(\bdiv;\Omega_h) \to \md{R}$, \ $\mc{B}: \bH(\bdiv;\Omega_h) \times L^2(\Omega_h) \to \md{R}$, and the functionals $\mc{F}_1 : \bH(\bdiv;\Omega_h) \to \md{R}$ and $\mc{F}_2 : L^2(\Omega_h) \to \md{R}$ are defined by
    \begin{subequations}
    \begin{align*}
    \mc{A}(\bq,\bv) &:= (\kap^{-1} \bq, \bv)_{\Omega_h}, \\
    \mc{A}_{T}(\bq,\bv) &:= \sum_{e\subset \partial \Omega_h } \int_e \left( \int_0^{l(\bx)} \kap^{-1} \bq(\bx + \boldsymbol{t}(\bx) s) \cdot \boldsymbol{t}(\bx) \right) \bv(\bx) \cdot \normalh \,ds \,dS_{\bx},  \\
    \mc{B}(\bq, w) &:= -(w, \nabla \cdot \bq)_{\Omega_h}, \\
    \mc{F}_{1}(\bv) &:=  - \langle\xi_0\circ\phi, \bv \cdot \normalh\rangle_{\partial \Omega_h}, \\
    \mc{F}_{2}( w) &:=  -(f, w )_{\Omega_h} .
    \end{align*}
    \end{subequations}
Beyond the difference in the domain of definition, the system above differs from the original problem \eqref{scheme-augmented-continuous} in the presence of the term $\mathcal A_T$, introduced by the transfer of boundary condition. The well posedness of problems of this form was established in \cite{OySoZu2019}. On the interest of brevity, we shall not repeat the argument here and instead will now discuss the discretization of this problem along with that of the integral equation \eqref{eq:IntegralProblem}.

\color{black}
%
\subsection{Discrete variational formulation}\label{sec:AugmentedHDG}
Having defined all the required notation, we can now state the HDG discretization of \eqref{eq:interior-problem no coupled} which, for Dirichlet data $\xi_0 \in H^{1/2}(\partial \Omega)$,  seeks an approximation $(\bq_h, u_h, \hat{u}_h)\in \bV_h \times W_h\times M_h$ satisfying  
	\begin{subequations}\label{HDG-method}
	\begin{align}
	(\boldsymbol{\kappa}^{-1} \bq_h, \bv)_{\mc{T}_h} - (u_h, \nabla \cdot \bv)_{\mc{T}_h} + \langle \hat{u}_h, \bv \cdot \normalh \rangle_{\partial \mc{T}_h}  &=  0, \label{eq:HDG-method_1} \\
	(\nabla\cdot \bq_h, w)_{\mc{T}_h} + \langle \tau\, u_h, w \rangle_{\partial \mc{T}_h} - \langle \tau \, \hat{u}_h, w \rangle_{\partial \mc{T}_h} &= (f, w)_{\mc{T}_h}, \label{eq:HDG-method_2} \\
	\langle \mu, \hat{\bq}_h \cdot \normalh \rangle_{\partial \mc{T}_h \setminus \partial \Omega_h} &= 0, \label{eq:HDG-method_3} \\
	\langle \hat{u}_h, \mu \rangle_{\partial \Omega_h} &= \langle \varphi^{\bq_h}_0, \mu \rangle_{\partial \Omega_h}, \label{eq:HDG-method_4}
	\end{align}
\noindent for any test $(\boldsymbol v, w, {\mu})\in \bV_h \times W_h\times M_h$. Following \cite{CoSo2012}, the approximate boundary data on $\partial \Omega_h$ appearing on the right hand side of \eqref{eq:HDG-method_4} is given by

    \begin{alignat}{6}
	 \varphi^{\bq_h}_0(\bx) &:= \xi_0\circ\phi(\bx) + \int_0^{l(\bx)} \boldsymbol{\kappa}^{-1} \, E\, \bq_h(\bx + \boldsymbol t(\bx) s) \cdot \boldsymbol t(\bx) \,ds \qquad&& \text{ for }\; \bx\in \partial \Omega_h. \label{def:varphi_0}
	\end{alignat}
	\end{subequations}
Where $E$ denotes the extrapolation operator. 
The numerical flux in the normal direction $\widehat{\boldsymbol q}_h\cdot \normalh$ is defined as 
    \begin{equation}\label{numerical-trace}
	\hat{\bq}_h \cdot \normalh = \bq_h \cdot\normalh + \tau\, (u_h - \hat{u}_h) \qquad \text{ on } \partial \mathcal{T}_h, 
	\end{equation}
{where $\tau$ stabilization function. Throughout this analysis we will only require $0<\tau\leq\overline{\tau}<\infty$, where $\overline{\tau}$ denotes the maximum value of $\tau$.} 

Note that, the terms $\langle \hat{u}_h, \bv \cdot\normalh \rangle_{\partial \mc{T}_h}$ and $\langle \tau \hat{u}_h, w \rangle_{\partial \mc{T}_h}$, given in \eqref{eq:HDG-method_1} and \eqref{eq:HDG-method_2}, respectively, can be split into the contributions of the interior edges and of the boundary edges as 
    \begin{subequations}
    \begin{align*}
    \langle \hat{u}_h, \bv \cdot \normalh \rangle_{\partial \mc{T}_h} &= \langle \hat{u}_h, \bv \cdot \normalh \rangle_{\partial \mc{T}_h  \setminus \partial\Omega_h}  + \langle \varphi^{\bq_h}_0, \bv \cdot \normalh \rangle_{\partial \Omega_h}, \\
    \langle \tau \hat{u}_h, w \rangle_{\partial \mc{T}_h} &= \langle \tau \, \hat{u}_h, w \rangle_{\partial \mc{T}_h \setminus \partial\Omega_h} +  \langle \tau \varphi^{\bq_h}_0, w \rangle_{\partial \Omega_h}.
    \end{align*}
    \end{subequations}
Replacing now the numerical flux \eqref{numerical-trace} in \eqref{eq:HDG-method_3}, results in
    \begin{equation*}
    \langle \mu, \bq_h \cdot \normalh \rangle_{\partial \mc{T}_h \setminus \partial\Omega_h} + \langle \mu, \tau (u_h - \hat{u}_h) \rangle_{\partial \mc{T}_h \setminus \partial\Omega_h}= 0.
    \end{equation*}

In order to apply known results from functional analysis, we rewrite the numerical trace $\hat{u}_h$ in terms of averages and jumps. For this, we use the equation \eqref{eq:HDG-method_3} and  separate the term featuring $\hat{u}_h$ as
    \begin{align*}
    0 &= \langle \mu, \bq_h \cdot \normalh \rangle_{\partial \mc{T}_h \setminus \partial\Omega_h} + \langle \mu, \tau u_h  \rangle_{\partial \mc{T}_h \setminus \partial\Omega_h} - \langle \mu, \tau \hat{u}_h \rangle_{\partial \mc{T}_h \setminus \partial\Omega_h} \\
    &= \sum_{T\in \mc{T}_h} \sum_{e\in \partial T \setminus \partial\Omega_h} \int_e \left(\mu\,  \bq_h\cdot \normalh + \tau\, \mu \, u_h - \tau\, \mu \, \hat{u}_h\right) \\
    &=  \sum_{e\in \mc{E}_h^{\circ}} \int_{e} \left( \jump{\bq_h} \,  \mu + 2\tau\, \prom{u_h} \, \mu - 2\tau\, \hat{u}_h \, \mu \right) 
    = \int_{\mc{E}_h^{\circ}} \left( \jump{\bq_h} + 2 \, \tau\, \prom{u_h} - 2\, \tau\, \hat{u}_h \right) \mu  \qquad \forall \, \mu \in M_h .
    \end{align*}
Above, we have used the fact that the hybrid variable $\hat{u}_h$ is single valued, and the average $\prom{\cdot}$ and jump $\jump{\cdot}$ operators are defined for every edge $e$ in a fashion analogous to \eqref{eq:JumpOperator} and \eqref{eq:jumpAndave}. Then, taking as test function $\mu =\jump{\bq_h} + 2 \, \tau\, \prom{u_h} - 2\, \tau\, \hat{u}_h \in M_h $ in the expression above, we deduce that
    \begin{equation*}
    \hat{u}_h = \frac{1}{2}\tau^{-1} \jump{\bq_h}  + \prom{u_h} \qquad \text{ on } \mc{E}_h^{\circ}.
    \end{equation*}

\noindent We make use of this identity to obtain
\begin{eqnarray*}
    \langle \hat{u}_h, \bv \cdot \normalh \rangle_{\partial \mc{T}_h  \setminus \partial\Omega_h} 
    = \langle  \hat{u}_h,\jump{\bv}  \rangle_{\mc{E}_h^{\circ}}
    = \frac{1}{2} \langle \tau^{-1} \jump{\bq_h}, \jump{\bv}\rangle_{\mc{E}_h^{\circ}} +
    \langle\prom{u_h}, \jump{\bv}\rangle_{\mc{E}_h^{\circ}}
 \end{eqnarray*}
 and
 \begin{align*}
    \langle \tau \, \hat{u}_h, w \rangle_{\partial \mc{T}_h \setminus \partial\Omega_h} 
    =
    2 \langle \tau \prom{w}, \hat{u}_h\rangle_{\mc{E}_h^{\circ}}
    = \langle\jump{\bq_h}, \prom{w}\rangle_{\mc{E}_h^{\circ}} +2 \langle\tau \prom{w}, \prom{u_h} \rangle_{\mc{E}_h^{\circ}}.
\end{align*}
In this way, replacing the definition of $\varphi_0^{\bq_h}$---see \eqref{def:varphi_0}---in \eqref{eq:HDG-method_1} and \eqref{eq:HDG-method_2}, together with the foregoing identities, we obtain that \eqref{HDG-method} is equivalent to finding  $(\bq_h,u_h)\in \bV_h \times W_h$ such that
    \begin{subequations}\label{scheme-augmented}
    \begin{align}
    \mc{A}_h(\bq_h, \bv) + \mc{A}_T(\bq_h,\bv) + \mc{B}_h(\bv,u_h) &= \mc{F}_{1,h}(\bv)  \qquad \forall \, \bv\in \bV_h, \label{Hdg discrete-1} \\
    \mc{B}_T(\bq_h,w) + \mc{B}_h(\bq_h,w) - \mc{C}_h(u_h,w ) &= \mc{F}_{2,h}(w) \qquad \forall \, w\in W_h, \label{Hdg discrete-2}
    \end{align}
    \end{subequations}
\color{black}
where the bilinear forms $\mc{A}_h: \bV_h \times \bV_h \to \md{R}$, $\mc{B}_h, \mc{B}_T: \bV_h \times W_h \to \md{R}$ , $\mc{C}_h : W_h \times W_h \to \md{R}$, and the functionals $\mc{F}_{1,h} : \bV_h \to \md{R}$ and $\mc{F}_{2,h}  : W_h \to \md{R}$ are defined by
    \begin{subequations}\label{def:operators}
    \begin{align}
    \label{eq:OperatorA} 
    \mc{A}_h(\bq_h,\bv) &:= (\kap^{-1} \bq_h, \bv)_{\mc{T}_h} + \frac{1}{2} \langle \tau^{-1} \jump{\bq_h}, \jump{\bv}\rangle_{\mc{E}_h^{\circ}},  \\
    \mc{A}_{T}(\bq,\bv) &:=\sum_{e\subset \partial \Omega_h } \int_e \left( \int_0^{l(\bx)} \kap^{-1} \bq(\bx + \boldsymbol{t}(\bx) s) \cdot \boldsymbol{t}(\bx) \right) \bv(\bx) \cdot \normalh \,ds \,dS_{\bx},  \\
    \mc{B}_h(\bq_h, w) &:= -(w, \nabla \cdot \bq_h)_{\mc{T}_h} + \langle \jump{\bq_h}, \prom{w} \rangle_{\mc{E}_h^{\circ}}\\
    \mc{B}_T(\bq_h, w) &:= \sum_{e\subset \Omega_h} \int_e \tau \left( \int_0^{l(\bx)} \kap^{-1} \bq_h(\bx + \boldsymbol{t}(\bx) s) \cdot \boldsymbol{t}(\bx)  \right) w(\bx)  \,ds \,dS_{\bx}, \\
    \label{eq:BilinearFormC}
    \mc{C}_h(u_h,w) &:= \langle \tau\, u_h, w \rangle_{\partial \mc{T}_h} -  2 \langle \tau\prom{u_h},\prom{w}\rangle_{\mc{E}_h^{\circ}}, \\
    \mc{F}_{1,h}(\bv) &:=  - \langle \xi_0\circ\phi, \bv \cdot \normalh \rangle_{\partial\Omega_{h}}, \\
    \mc{F}_{2,h}( w) &:= - (f, w )_{\mc{T}_h} - \langle \tau \,  \xi_0\circ\phi, w \rangle_{\partial\Omega_{h}}.
    \end{align}
    \end{subequations}

The unique solvablity of the scheme \eqref{scheme-augmented} will be proved by an energy argument. To that end, for $e\in \partial \Omega_h$ and $\bv \in \bL^2(T_e^{\text{ext}})$, it is convenient to define the following norm on the extension patch $T_e^{\text{ext}}$:
\begin{align*}
    	\vertiii{\bv}_e:=\,&\left(\int_{e}\int_{0}^{l(\boldsymbol{x})}|\bv(\boldsymbol{x}+s\boldsymbol{t}(\boldsymbol{x}))|^{2}\,ds\,dS_{\boldsymbol{x}}\right)^{1/2}.
\end{align*}

This norm is equivalent to the standard $\bL^2(T_e^{\text{ext}})$-norm as shown first in \cite{OySoZu2019} for the two dimensional and later extended to three dimensions in \cite{OySoZu2020}. That is, there exist positive constants $C_1^e$ and $C_2^e$, independent of $h$, such that,
\begin{equation}\label{norm-equivalence}
C_1^e 	\vertiii{\bv}_e \leq \|\bv\|_{T_e^{\text{ext}}} \leq C_2^e 	\vertiii{\bv}_e.
\end{equation}

This equivalence holds true under certain conditions on the transferring vectors $\bt(\bx)$ (cf. \cite{OySoZu2020,OySoZu2019})) ensuring, roughly speaking, that they cannot deviate too much from the vector normal to $e$.

We also introduce the element-wise constants
    \begin{equation}\label{eq:CextCinv}
	C^e_{ext} := \dfrac{1}{\sqrt{r_e}} \sup_{\boldsymbol{\chi}\in \mathcal V^k} \dfrac{\vertiii{\boldsymbol{\chi}}_e }{\|\boldsymbol{\chi}\|_{T_e}} \quad \text{ and } \quad
	C^e_{inv} := h_e^{\perp} \sup_{\boldsymbol{\chi}\in \mathcal V^k} \dfrac{\vertiii{\nabla \boldsymbol{\chi}} \|_{T_e} }{\|\boldsymbol{\chi} \|_{T_e}}, \end{equation}
where 
$ \mathcal V^k := \left\{ \boldsymbol p \in \mathds [\mathds{P}_k(T_e^{ext} \cup T_e)]^2\, : \,  \boldsymbol p \neq\boldsymbol 0 \right\}
$. These constants are independent of $h$, but depend on the polynomial degree $k$ and the mesh regularity parameter as shown in \cite{CoQiuSo2014}.

We now proceed to derive an energy inequality that will lead to the well-posedness of \eqref{scheme-augmented}.

\begin{lem}\label{lemma:energy}
Let $\alpha_h=R_h\underline{\kap}^{-1}(\overline{\kap}
-\overline{\kap}^{1/2}h^{1/2} \overline{\tau}^{1/2})$ and $\beta_h=\underline{\kap}^{-1}\overline{\kap}^{1/2}R_h h^{1/2} \overline{\tau}^{1/2}$. It holds
\begin{align}\label{ineq:energy}
(1-\alpha_h)\|\kap^{-1/2}\bq_h\|_{0,\Omega_h}^2 &+(1-\beta_h) \|\tau^{1/2}u_h\|_{\partial \Omega_h}^2
+\|\tau^{1/2} (u_h - \prom{u_h})\|_{\partial \mc{T}_h\setminus \partial \Omega_h}^2
+ \|\tau^{-1/2} \jump{\bq_h}\|_{\mc{E}_h^\circ}^2\nonumber\\
&\lesssim \|\kap^{1/2} h^{-1/2}\xi_0\circ\phi\|_{\partial\Omega_{h}}^2
 + \|f\|_{0,\Omega} \|u_h\|_{0,\Omega_h}.
\end{align}

\end{lem}
\begin{proof}
 By taking $\bv=\bq_h$ and $w=u_h$ in \eqref{scheme-augmented}, and subtracting the resulting expressions we obtain
\begin{align}
\|\kap^{-1/2}\bq_h\|_{0,\Omega_h}^2 &+ \frac{1}{2}\|\tau^{-1/2} \jump{\bq_h}\|_{\mc{E}_h^\circ}^2 +\mc{A}_T(\bq_h,\bq_h)
+\mc{B}_T(\bq_h, u_h) + \mc{C}_h(u_h,u_h)\nonumber\\
&= - \langle \xi_0\circ\phi, \bq_h \cdot \normalh \rangle_{\partial\Omega_{h}}
 - (f, u_h )_{\mc{T}_h} - \langle \tau \,  \xi_0\circ\phi, u_h \rangle_{\partial\Omega_{h}}\label{eq:aux-energy}.
\end{align}
First of all, after performing algebraic calculations, we observe that $\mc{C}$ is a semi-definite operator from $W_h\times W_h$ to $\mathbb{R}$. In fact, 
\begin{align}\label{Ch-semidefinite}
\mc{C}_h(u_h,u_h)=\langle \tau\, u_h, u_h \rangle_{\partial \mc{T}_h} - 2 \langle \tau\prom{u_h},\prom{u_h}\rangle_{\mc{E}_h^{\circ}}
= \|\tau^{1/2} (u_h - \prom{u_h})\|_{\partial \mc{T}_h\setminus \partial \Omega_h}^2
+ \|\tau^{1/2} u_h\|_{ \partial \Omega_h}^2.
\end{align}
We will now obtain a lower bound for the non-positive terms of left hand side of \eqref{eq:aux-energy}. In this direction, the operator $\mc{A}_T$ can be bounded as follows. Let $e\in \subset \partial \Omega_h$ and $\bx\in e$. By the Cauchy-Schwarz inequality and the definition in \eqref{eq:CextCinv},
\begin{align*}
 \int_0^{l(\bx)} \kap^{-1} \bq_h(\bx + \boldsymbol{t}(\bx) s) \cdot \boldsymbol{t}(\bx) ds
 \leq l(\bx)^{1/2} \vertiii{\kap^{-1} \bq_h}_e
\leq  h_{T_e}^{1/2} r_e \underline{\kap}^{-1}\overline{\kap}^{1/2}C_{ext}^e \|\kap^{-1/2}\bq_h\|_{T_e},
\end{align*}
where we have used the bound $l(\bx)\leq h_{T^e} r_e$. Then, by the discrete trace inequality, we have
\begin{align}\label{boundAT}
-\mc{A}_T(\bq_h,\bq_h) \leq& |\mc{A}_T(\bq_h,\bq_h)|
\lesssim \underline{\kap}^{-1}\overline{\kap}^{1/2}\sum_{e\subset \partial \Omega_h} h^{1/2}_{T_e} r_h  \|\kap^{-1/2}\bq_h\|_{e}\|\bq_h\cdot\normalh\|_{e}\nonumber\\
\lesssim& R_h \underline{\kap}^{-1}\overline{\kap}\|\kap^{-1/2}\bq_h\|_{\Omega_h}^2.
\end{align}
The same arguments yield to 
\begin{align}\label{boundBT}
-\mc{B}_T(\bq_h, u_h) \leq& |\mc{B}_T(\bq_h, u_h)| \lesssim \underline{\kap}^{-1}\overline{\kap}^{1/2} R_h h^{1/2} \overline{\tau}^{1/2}
\|\kap^{-1/2}\bq_h\|_{0,\Omega_h} \|\tau^{1/2}u_h\|_{0,\partial\Omega_h}\nonumber\\
 \leq& \underline{\kap}^{-1}\overline{\kap}^{1/2}R_h h^{1/2} \overline{\tau}^{1/2}
\left(\frac{1}{2}\|\kap^{-1/2}\bq_h\|_{0,\Omega_h}^2 +\frac{1}{2}\|\tau^{1/2}u_h\|_{0,\partial\Omega_h}^2\right).
\end{align}
Therefore, combining the above estimates and \eqref{eq:aux-energy}, we deduce that
\begin{align*}
(1-R_h\underline{\kap}^{-1}\overline{\kap}
&-\underline{\kap}^{-1}\overline{\kap}^{1/2}R_h h^{1/2} \overline{\tau}^{1/2})\|\kap^{-1/2}\bq_h\|_{0,\Omega_h}^2 +(1-\underline{\kap}^{-1}\overline{\kap}^{1/2}R_h h^{1/2} \overline{\tau}^{1/2}) \|\tau^{1/2}u_h\|_{\partial \Omega_h}^2\\
&+\|\tau^{-1/2} \jump{\bq_h}\|_{\mc{E}_h^\circ}^2+\|\tau^{1/2} (u_h - \prom{u_h})\|_{\partial \mc{T}_h\setminus \partial \Omega_h}^2\\
&\lesssim | \langle \xi_0\circ\phi, \bq_h \cdot \normalh \rangle_{\partial\Omega_{h}}|
+|  (f, u_h )_{\mc{T}_h}| +| \langle \tau \,  \xi_0\circ\phi, u_h \rangle_{\partial\Omega_{h}}|.
\end{align*}

\color{black}
Finally, the result follows by the discrete trace inequality applied to the boundary terms on the right hand side, Young's inequality and the definition of $\alpha_h$ and $\beta_h$.

\end{proof}
\begin{crl}
The HDG scheme \eqref{scheme-augmented} is well-posed for $h$ sufficiently small.
\end{crl}
\begin{proof}
Let $f\equiv 0$ and $\xi_0=0$. By \eqref{ineq:energy} we obtain that $\bq_h=\boldsymbol{0}$. Moreover, since $\tau>0$, we have that $u_h=0$ on the boundary $\Omega_h$ and $u_h = \prom{u_h}$ on $\partial \mc{T}_h$; therefore $u_h$ is continuous. These facts, together with \eqref{Hdg discrete-2} lead to
\begin{align*}
0 =& -(u_h, \nabla \cdot \bv)_{\mc{T}_h} + \langle \jump{\bv}, u_h\rangle_{\mc{E}_h^{\circ}}  =(\nabla u_h,\bv) \qquad \forall \, \bv\in \bV_h.
\end{align*}
Thus, taking $\bv =\nabla u_h$ we conclude that $u_h = 0$ since it vanishes at the boundary.
\end{proof}

The energy estimate in Lemma \ref{lemma:energy} provides the stability bound for the vector-valued unknown $\bq_h$. On the other hand, the stability for the scalar approximation $u_h$ can be obtained by a duality argument that we omit since it is not need it for the analysis of the coupled problem. We refer the reader to the proof of Lemma 3.5 in \cite{CoQiuSo2014} or the proof of Theorem 3.1 in \cite{SoVa2022} for details regarding the duality argument employed in this type of unffited HDG methods. Therefore, it is possible to conclude that there is a constant $C_{HDG}> 0$, independent of $h$, such that

\begin{align}\label{ estim HDG}
 \J(\mbf{q}_h,u_h)   +\|u_h\|_{\Omega_h} \leq C_{HDG}\, \left( \| f\|_{\Omega_h} +  \|\kap^{1/2}h^{-1/2}\xi_0\circ \phi\|_{\partial\Omega_h} \right),
    \end{align}
where, for convenience of notation of the forthcoming analysis, we have denoted
\begin{equation}\label{def:J}
\J(\mbf{q}_h,u_h):=  \left(  \|\kap^{-1/2}\bq_h\|_{\Omega_h}^2 + \|\tau^{1/2}u_h\|_{\partial \Omega_h}^2
+\|\tau^{1/2} (u_h - \prom{u_h})\|_{\partial \mc{T}_h\setminus \partial \Omega_h}^2
+ \|\tau^{-1/2} \jump{\bq_h}\|_{\mc{E}_h^\circ}^2\right)^{1/2}.
\end{equation}

Having established the well posedness of the discrete formulation, in the following section we will study the behavior of the discretization error.

\subsection{\textit{A priori} error analysis}
To establish a priori error bounds for the HDG discretization we will make use of a tool introduced by Francisco--Javier Sayas, Jay Gopalakrishnan and Bernardo Cockburn in \cite{CoGoSa2010}. The idea is to use a projection, known as the \textit{HDG projection}, to decompose the discretization into a component involving the approximation properties of the discrete spaces $\boldsymbol V_h$ and $W_h$, and another component involving the error introduced by projecting into these spaces. 
The HDG projection over $\mbf{V}_h\times W_h$, denoted by $\bsy{\Pi}(\mbf{q},u):=(\bsy{\Pi}_{\mathrm v}\mbf{q},\Pi_{\mathrm w} u)$, is the unique element-wise solution pair of
	\begin{subequations}\label{eq:HDGprojector}
	\begin{align}
	(\bsy{\Pi}_{\mathrm v}\mbf{q}, \mbf{v})_T &= (\mbf{q}, \mbf{v})_T &  &\forall \ \mbf{v} \in [\md{P}_{k-1}(T)]^e, \label{properties projector Pi_v} \\
	(\Pi_{\mathrm w} u, w)_T &= (u,w)_T & &\forall \ w\in \md{P}_{k-1}(T), \label{properties projector Pi_w} \\
	\pdual{\bsy{\Pi}_{\mathrm v}\mbf{q}\cdot \mbf{n} + \tau \Pi_{\mathrm w} u, \mu}_{e} &= \pdual{\mbf{q} \cdot \mbf{n} + \tau u, \mu}_e & &\forall \ \mu \in \md{P}_k(e), \label{properties projector Pi_h}
	\end{align}
    \end{subequations}
for every element $T\in \mc{T}_h$, and $e\subset \partial T$. The approximation properties of $\bsy{\Pi}$ are stated in Section \ref{app:HDG-proj}. Using this projection we can then define
    \begin{equation*}
\errorq := \boldsymbol{\Pi}_{\bV}\bq - \bq_h \qquad \erroru:= \Pi_W u - u_h \qquad \text{ and } \qquad     \bI^{\bq}:= \bq-\boldsymbol{\Pi}_{\bV}\bq  \qquad I^u:=u - \Pi_W u,
    \end{equation*}
where $\boldsymbol \Pi_{\boldsymbol V}$ is the HDG projector onto $\mathbf V_h$, and $\Pi_W$ is the HDG projector onto $W_h$. The terms $\boldsymbol \varepsilon^{\boldsymbol q}$ and $\varepsilon^u$ are known as \textit{the projections of the errors} and the terms $\bI^{\bq}$ and $I^u$ are the \textit{errors of the projections}. The full discretization error can then be split as 
    \begin{equation*}
    \bq-\bq_h = \errorq +\bI^{\bq} \quad \text{ and } \quad u-u_h =  \erroru + I^u.
    \end{equation*}
We will now show that the scheme \eqref{scheme-augmented} is consistent and the discretization error is driven solely by the approximation properties of the discrete spaces, as encoded by $\boldsymbol I^{\boldsymbol q}$, and $I^{u}$. We start by noting that from \eqref{Hdg discrete-1} and the decompositions above, it follows that
    \begin{equation}\label{eq:AVF}
    \mc{A}_h(\bq- \errorq - \bI^{\bq}, \bv) 
    +\mc{A}_T(\bq- \errorq - \bI^{\bq},\bv)+ \mc{B}_h(\bv,u- \erroru - I^u) = \mc{F}_{1,h}(\bv) \qquad \forall \bv \in \bV_h.
    \end{equation}
However, since $\bq$ and $u$ satisfy \eqref{eq:interior-problem no coupled} in a distributional sense, we have that $\bq\in \bH(\bdiv;\Omega_h)$ and therefore $\jump{\bq}=0$ in $\mc{E}_h^{\circ}$. This also implies that $u\in H^1(\Omega_h)$ since $\nabla u = -\kap^{-1}\bq \in L^2(\Omega_h)$. Hence,
\begin{align*}
    \mc{A}_h(\bq, \bv) +&
    \mc{A}_T(\bq,\bv)+
    \mc{B}_h(\bv,u) - \mc{F}_{1,h}(\bv) \\=&
    (\kap^{-1} \bq, \bv)_{\mc{T}_h}  
     + \sum_{e\subset \partial \Omega_h } \int_e \left( \int_0^{l(\bx)} \kap^{-1} \bq(\bx + \boldsymbol{t}(\bx) s) \cdot \boldsymbol{t}(\bx) \right) \bv(\bx) \cdot \normalh \,ds \,dS_{\bx}\\
     &+\langle \jump{\bv}, \prom{u_h} \rangle_{\mc{E}_h^{\circ}}-(u, \nabla \cdot \boldsymbol v)_{\mc{T}_h}+ \langle \xi_0\circ\phi, \bv \cdot \normalh \rangle_{\partial\Omega_{h}}  \\
     =&
    (\kap^{-1} \bq, \bv)_{\mc{T}_h}  
     +\langle \jump{\bv}, \prom{u} \rangle_{\mc{E}_h^{\circ}}-(u, \nabla \cdot \boldsymbol v)_{\mc{T}_h}+ \langle u, \bv \cdot \normalh \rangle_{\partial\Omega_{h}},
\end{align*}
where in the last equality we have used the fact that $\bq$ satisfies the transfer equation \eqref{def:xi_o-tilde} and $u$ satisfies \eqref{eq:interior3}. Then, by integrating by parts and considering equation \eqref{eq:interior2}, we obtain that
\begin{align*}  
  \mc{A}_h(\bq, \bv) +
    \mc{A}_T(\bq,\bv)+
    \mc{B}_h(\bv,u) - \mc{F}_{1,h}(\bv) =&\langle \jump{\bv}, \prom{u} \rangle_{\mc{E}_h^{\circ}}
  -\langle u,\bv \cdot \normalh\rangle_{\partial \mc{T}_h} + \langle u, \bv \cdot \normalh \rangle_{\partial\Omega_{h}} = 0.
\end{align*}

    Analogously, from \eqref{Hdg discrete-2} we have
\begin{equation}\label{eq:BBCF}
        \mc{B}_T(\bq - \errorq - \bI^{\bq},w) + \mc{B}_h(\bq - \errorq - \bI^{\bq},w) - \mc{C}_h(u- \erroru - I^u,w ) = \mc{F}_{2,h}(w).
\end{equation}
Analyzing the terms above that involve $\boldsymbol q\in \bH(\bdiv;\Omega_h)$ and $u\in H^1(\Omega_h)$, and using again the facts that $\bq$ satisfies the transfer equation \eqref{def:xi_o-tilde} and $u$ satisfies \eqref{eq:interior3}, it is easy to verify that
\begin{align*}
 \mc{B}_T(\bq,w) + \mc{B}_h(\bq,w)-\mc{C}_h(u,w) -\mc{F}_{2,h}(w) 
 =&- \langle \tau\, u, w \rangle_{\partial \mc{T}_h}  + 2 \langle \tau\, u,\prom{w}\rangle_{\mc{E}_h^{\circ}}
 + \langle \tau\, u, w \rangle_{\partial \Omega_h} =0
\end{align*}
%
Putting these arguments together it follows from \eqref{eq:AVF} and \eqref{eq:BBCF} that the scheme is consistent and the following error equations for $( \errorq ,\erroru )\in\bV_h\times W_h $ hold
\begin{subequations}
    \begin{align*}
     \mc{A}_h(\errorq, \bv) 
     +\mc{A}_T(\errorq, \bv) +
     \mc{B}_h(\bv,\erroru) &=  -\mc{A}_h(\bI^{\bq}, \bv) 
     -\mc{A}_T(\bI^{\bq}, \bv) +
     \mc{B}_h(\bv,I^u),\\
    \mc{B}_T(\errorq,w) +\mc{B}(\errorq,w) - \mc{C}(\erroru,w ) &=  - \mc{B}_T(\bI^{\bq},w) - \mc{B}_h(\bI^{\bq},w) + \mc{C}_h(I^{u}, w), 
\end{align*}
$\forall (\bv,w)\in \bV_h\times W_h $.
\end{subequations}
%
%
Now, by the orthogonality properties of the HDG projection \eqref{eq:HDGprojector},
we deduce that
\[
\mc{A}_h(\bI^{\bq}, \bv) 
   +
     \mc{B}_h(\bv,I^u) = (\boldsymbol{\kappa}^{-1}\, \bI^{\bq}, \bv)_{\mc{T}_h}
\]
and 
\[
 - \mc{B}_h(\bI^{\bq},w) + \mc{C}_h(I^{u}, w) = 0.
 \]
In this way, we conclude that the projection of the errors  $( \errorq ,\erroru )\in\bV_h\times W_h $ satisfy
\begin{subequations}
\label{error-scheme2}
    \begin{align}
     \mc{A}_h(\errorq, \bv) 
     +\mc{A}_T(\errorq, \bv) +
     \mc{B}_h(\bv,\erroru) &= \mc{G}_1 (\bv),\\
    \mc{B}_T(\errorq,w)+ \mc{B}(\errorq,w) - \mc{C}(\erroru,w ) &=  \mc{G}_2(w).
\end{align}
$\forall (\bv,w)\in \bV_h\times W_h $,
\end{subequations}
with
\[
\mc{G}_1(\bv): = -(\boldsymbol{\kappa}^{-1}\, \bI^{\bq}, \bv)_{\mc{T}_h}  
     -\mc{A}_T(\bI^{\bq}, \bv)
\]
and
\[
\mc{G}_2(w): =  - \mc{B}_T(\bI^{\bq},w).
\]

\begin{thm}\label{A1:thm:estim normH eroor}
For $h$ sufficiently small, there hold
\begin{align}\label{error_estimate_q}
\|\kap^{-1/2}(\mbf{q}-\mbf{q}_h)\|_{\Omega_h}
\lesssim\,& 
 \|\kap^{-1/2}\bI^{\bq}\|_{\Omega_h}
        +  (R_h \underline{\kap}^{-2} \overline{\kap}+\overline{{\tau}} )^{1/2}  \|\kap^{-1/2}\bI^{\bq}\|_{\Omega_h^c}.       
\end{align}
Moreover, under elliptic regularity it holds
\begin{align}\label{error_estimate_u}
\|u - u_h\|_{\Omega_h}
\lesssim &\; 
\left(h+
 \left(h\overline{\tau}^{1/2}+h^{1/2}\right)R_h\right)\|\kap^{-1/2}\bI^{\bq}\|_{\Omega_h}\nonumber\\
 \,&  + \left(h^{1/2}\, \overline{\tau}^{1/2}R_h+1\right) \|I^{u} \|_{\Omega_h}
        + R_h\left(\overline{\tau}^{1/2} +h^{1/2}\right) \|h\partial _{\bt}(\bI^{\bq}\cdot\bt)\|_{\Omega_h^c}.        
\end{align}
\end{thm}

\begin{proof}
By proceeding exactly as in the proof of Lemma \ref{lemma:energy}, but in the context of the equation of the projection of the errors \eqref{error-scheme2}, for $h$ sufficiently small, we deduce that
\begin{align*}
\J(\errorq,\erroru)\lesssim  |\mc{G}_1(\errorq)|  +  |\mc{G}_2(\erroru)|,
\end{align*}
where we recall the definition of $\J$ in \eqref{def:J}.
\color{black}
In order to bound the terms on the right-hand side, we employ the Cauchy-Schwarz and discrete trace inequalities and obtain that
    \begin{align*}
        |\mc{G}_1(\errorq)|  \leq&  \|\kap^{-1/2}\bI^{\bq}\|_{\Omega_h}\|\kap^{-1/2}\errorq\|_{\Omega_h}  
        + \underline{\kap}^{-1}\sum_{e \subset \partial \Omega_h} \vertiii{\bI^{\bq}}_e \, \|l^{1/2}\,\errorq \cdot \bn\|_{e} \\
       \lesssim & 
        \left( \|\kap^{-1/2}\bI^{\bq}\|_{\Omega_h}^2+  R_h \underline{\kap}^{-2} \overline{\kap} \sum_{e \subset \partial \Omega_h} \vertiii{\bI^{\bq}}_e^2\right)^{1/2} \|\kap^{-1/2}\errorq\|_{\Omega_h}^2
    \end{align*}
where we have also used the fact that $l(\bx)\lesssim R_h h$ for all $\bx\in \partial \Omega_h$.    
Similarly,
   \begin{align*}
        |\mc{G}_2(\erroru)|  \leq& \left(\sum_{e \subset \partial \Omega_h} \vertiii{\tau^{1/2}\bI^{\bq}}_e^2\right)^{1/2} \|\tau^{1/2} \erroru\|_{\partial\Omega_h}.
   \end{align*}
Therefore, by combining the the above inequalities, we obtain
\begin{align*}
\J(\errorq,\erroru)^2 \lesssim 
 \|\kap^{-1/2}\bI^{\bq}\|_{\Omega_h}^2
        +  (R_h \underline{\kap}^{-2} \overline{\kap}+\tau) \sum_{e \subset \partial \Omega_h} \vertiii{\bI^{\bq}}_e^2
\end{align*}
and \eqref{error_estimate_q} follows by the fact that $\|\kap^{-1/2}(\mbf{q}-\mbf{q}_h)\|_{\Omega_h}\leq \|\kap^{-1/2}\bI^{\bq}\|_{\Omega_h}+\J(\errorq,\erroru)$ and the norm equivalence \eqref{norm-equivalence}.
On the other hand, by a duality argument (Lemma 3.9 in \cite{CoQiuSo2014}), it is possible to derive that
\begin{align*}
\|\erroru\|_{\Omega_h} \lesssim 
 (h+(h\overline{\tau}^{1/2}+h^{1/2})R_h)\|\kap^{-1/2}\bI^{\bq}\|_{\Omega_h}
 + \overline{\tau}^{1/2}h^{1/2} R_h\|I^{u} \|_{\Omega_h}
        +  R_h(h^{1/2} +\tau^{1/2}) \|h\partial _{\bt}(\bI^{\bq}\cdot\bt)\|_{\Omega_h^c},
\end{align*}
which implies \eqref{error_estimate_u}.
\end{proof}
\begin{crl} If $(\bq,u)\in\bH^{k+1}(\Omega)\times H^{k+1}(\Omega)$ and $\tau$ is of order one, then
\begin{align}
\|\kap^{-1/2}(\bq - \bq_h)\|_{\Omega_h} +\|u-u_h\|_{\Omega_h}  \lesssim  h^{k+1} \left(|\bq|_{k+1,\Omega}+|u|_{k+1,\Omega}\right).
\end{align}

Moreover, if $(\bq,u)\in\bH^{1}(\Omega)\times H^{1}(\Omega)$, then 
\begin{align}\label{ineq:Jqu}
\J(\bq-\bq_h,u-u_h)
\lesssim& (\underline{\kap}^{-1/2}+\underline{\tau}^{-1/2}h^{1/2}+1) |\bq|_{1,\Omega}
+(\overline{\tau}^{1/2}+1)\overline{\tau}^{1/2} |u|_{1,\Omega}.
\end{align}
\end{crl}

\begin{proof}
The first inequality follows from the approximation properties of the HDG projection stated in Section \ref{app:HDG-proj}. On the other hand, 
\begin{align*}
\J(\bq-\bq_h,u-u_h)\leq&   \J(\errorq,\erroru) + \J(\bI^{\bq},I^u)\\  
\lesssim&\|\kap^{-1/2}\bI^{\bq}\|_{\Omega_h}
        +  (R_h \underline{\kap}^{-2} \overline{\kap}+\overline{{\tau}} )^{1/2}  \|\kap^{-1/2}\bI^{\bq}\|_{\Omega_h^c}+ \J(\bI^{\bq},I^u).   
\end{align*}
But, using the approximation estimates \eqref{C0:eq:projection_error}, we have
\begin{align*}
\J(\bI^{\bq},I^u)^2= & \|\kap^{-1/2}\bI^{\bq}\|_{\Omega_h}^2 + \|\tau^{1/2}I^u\|_{\partial \Omega_h}^2
+\|\tau^{1/2} (I^u - \prom{I^u})\|_{\partial \mc{T}_h\setminus \partial \Omega_h}^2
+ \|\tau^{-1/2} \jump{\bI^{\bq}}\|_{\mc{E}_h^\circ}^2\\
\lesssim& (\underline{\kap}^{-1}+\underline{\tau}^{-1}h^{-1}) h^2|\bq|_{1,\Omega}^2
+(\overline{\tau}+1)\overline{\tau} h^2|u|_{H^{1}(\Omega)}^2
+h^2|\nabla \cdot \bq|_{1,\Omega}^2
\end{align*}
and \eqref{ineq:Jqu} follows.
\color{black}
\end{proof}

\section{BEM discretization of the exterior problem}\label{sec:BEM}

For the discretization of the integral equation \eqref{eq:IntegralProblem} we will take advantage of the fact that the parametrization of artificial boundary $\Gamma$ is smooth and does not intersect with the support of the source term. It is a standard result in potential theory that these two conditions imply that the densities $\lambda$ and $g$ are both $C^\infty$, which allows for a simple, spectrally convergent discretization using interpolating trigonometric polynomials---an idea that had been implemented in \cite{MaMe2002} coupled with the finite element method over curved triangulations. For two dimensional problems, an exhaustive account of the theory of periodic boundary integral equations and their approximation can be found in the monograph by Saranen and Vainikko \cite{SaVa2002}. Here we will present only those basic results that will be used for the coupled formulation that will be described later.

If we let $\boldsymbol y$ be a $2\pi-$periodic, $C^\infty$ parametrization of $\Gamma$ such that $|\boldsymbol y^\prime(\cdot)|>0$ and $t\neq s\in[0,2\pi)$ implies that $\boldsymbol y(t)\neq \boldsymbol y(s)$, then the integral operators appearing in \eqref{eq:IntegralProblem} can be written in parametric form as
\[
\mathcal V g (\boldsymbol x(t)) = \int_{0}^{2\pi} V(s,t) \lambda\circ\boldsymbol y(s) |\boldsymbol y^\prime(s)|ds \quad \text{ and } \quad \mathcal K g (\boldsymbol x(t)) = \int_{0}^{2\pi} K(s,t) \varphi\circ\boldsymbol y(s) |\boldsymbol y^\prime(s)|ds.
\]
Where the integral kernels are the 2D Green function for the minus Laplacian and its normal derivative, namely
\[
V(s,t) : = -\frac{1}{2\pi}\log\left|\boldsymbol y(s)- \boldsymbol y(t)\right| \quad \text{ and } \quad K(s,t) : = \frac{1}{2\pi}\frac{\left(\boldsymbol y(s)- \boldsymbol y(t)\right)\cdot\boldsymbol n(\boldsymbol y(s))}{\left|\boldsymbol y(s)- \boldsymbol y(t)\right|^2}.
\]
The idea is then to discretize the parameterization of $\Gamma$ into $2n$ equispaced points  $t_0,\ldots,t_{2n-1} \in [0,2\pi)$ and use these points as interpolation nodes to collocate equation \eqref{integral-identity}. Due to the periodicity, it is natural to use trigonometric polynomials as a basis, and we will now introduce two spaces of trigonometric polynomials 
    \begin{equation*}
    \md{T}_n := \left\{ \sum_{j=0}^n a_j \cos(jt) + \sum_{j=1}^{n-1} b_j \sin(jt)  : a_j, b_j \in \md{R} \right\} , \quad \text{ and } \quad \md{T}_n^0 := \left\{ \lambda_n \in \md{T}_n : \int_0^{2\pi} \lambda_n\, ds = 0 \right\}.
    \end{equation*}
For real numbers $p \leq q$ and any function $\lambda\in H^q(0,2\pi)$ the space $\mathbb T_n$ has the following approximation property \cite{CaQu1982}
\begin{equation*}
\|\lambda - \mc{P}\lambda\|_{H^p(0,2\pi)} \leq (n/2)^{p-q}\|\lambda\|_{H^q(0,2\pi)},
\end{equation*}
where $\mc{P}$ is the $L^2$ projector onto $\mathbb T_n$. The Lagrangian basis for interpolation in $\mathbb T_n$ is given by
\[
L_j(t) := \frac{1}{2n}\left(1 + 2\sum_{k=1}^{n-1}\cos\left(k(t-t_j)\right) + \cos\left(n(t-t_n)\right)\right) \quad \text{for }\; j=0,1,\ldots,2n-1.
\]
These functions can be used to build the basis for $\mathbb T_n^0$, which is given by the set \[
\left\{ L_j - L_0: j =0,1,\ldots, 2n-1 \right\}.
\]
If we denote by $\mathbb Q_n^0$ the interpolation operator over $\mathbb T_n^0$, the following estimate holds \cite{SaVa2002} for $q >1/2$ and $0\leq p \leq q$:
\begin{equation*}
\|u - \mathbb Q_n^0 u\|_{H^p(0,2\pi)} \leq c_q (n/2)^{p-q}\|u\|_{H^q(0,2\pi)},
\end{equation*}
where $c_q = \left(1 + \sum_{j=1}^{\infty}\frac{1}{j^{2q}}\right)$. Therefore, if $\lambda: \Gamma \to \md{R}$ is known, the discrete version of the problem \eqref{eq:IntegralProblem} becomes that of finding $g_n$ such that $g_n \circ \by | \by'(\cdot)| \in \md{T}_n^0$, and 
    \begin{equation}\label{int-identity_app}
      \int_0^{2 \pi} \left( \frac{1}{2} g_n - \mc{K} g_n\right) (\by(s)) \psi(s) ds = -\int_0^{2\pi} (\mc{V}\, \lambda) (\by(s)) \psi(s) ds  \qquad \forall \, \psi \in \md{T}_n^0. 
    \end{equation}

Note that the term involving the constant $u_\infty$ drops out of the formulation when testing with $\psi\in\mathbb T_n^0$.  To determine $u_\infty$ we go back to \eqref{int-identity_app} and notice that we can \textit{define} an approximation $u_\infty^n$ to $u_\infty$ by testing with any $\varrho\in \mathbb  T_n\setminus\mathbb T_n^0$. Setting $\varrho=1$ then leads to
\[
     u_\infty^n :=\frac{1}{2\pi}\left( \int_0^{2\pi} (\mc{V}\, \lambda)(\by(s)) ds +  \int_0^{2 \pi} \left( \frac{1}{2} g_n - \mc{K} g_n\right)(\by(s)) ds  \right).
\]
    
Hence, we first solve \eqref{int-identity_app} for $\lambda_n$ and then  fix the value of $u_\infty^n$ by means of the definition above. It is clear that as the approximation $\lambda_n$ converges, the value of $u_\infty^n$ will converge as well. Pertaining the well-posedness of the discrete integral equation, it is pointed out that (as shown in \cite[Sec. 6.3--6.5]{SaVa2002}) the periodic operator $\mathcal V$ is a Fredholm operator of index $0$ over the periodic space 
\[
H^{-1/2}_0(0,2\pi):=\left\{\lambda \in H^{-1/2}(0,2\pi): \int_{0}^{2\pi}\lambda ds = 0\right\},
\]
from which the unique solvability of \eqref{int-identity_app} follows. Moreover, for a Galerkin approximation of  \eqref{int-identity_app} it can be shown \cite[Thm. 9.4.1]{SaVa2002} that the following error estimate holds
\[
\|g - g_n\|_{H^p(0,2\pi)} \leq c_q n^{p-q}\|g\|_{H^q(0,2\pi)} \quad \text{for }\; q-p>1/2.
\]
Combining this approximation result with the stability estimate \eqref{ineq:ContDependenceBIE} and the boundedness of the single layer operator $\mathcal V$ we arrive at
\begin{equation*}
\|g - g_n\|_{H^p(0,2\pi)} \leq C_{q,\mathcal V}\, n^{p-q} \|\lambda\|_{H^q(0,2\pi) }  \quad \text{for }\; q-p>1/2.
\end{equation*}

\section{Iterative coupled procedure}\label{sec:Coupling}
In \textit{Coupling at a distance}, the authors proposed an iterative method to find the solution to the original problem\eqref{eq:ext-diff-problem} by alternating between the solutions of the interior and exterior problems using 
HDG and spectral BEM respectively. The idea can be traced back to \cite{CoSa2011} and involves using the Dirichlet trace of $u$ over the artificial boundary as the unknown coupling variable and alternating between the solution of an interior and an exterior problem.

We start by observing that, from the discrete version of the transmission condition \eqref{eq:ContinuityQN}
    \begin{equation*}
    \int_0^{2\pi} \bq_h(\bx(s)) \cdot \bn(\bx(s)) \  \eta(s) ds + \int_0^{2\pi} \lambda_n(\bx(s)) \eta(s) ds  = 0 \qquad \forall \, \eta \in \md{T}_n^0,     
    \end{equation*}
the Neumann trace of the exterior problem can be written in terms of its interior counterpart as
    \begin{equation}\label{lambda-qh}
    \lambda_n = - \mc{P}(\bq_h\cdot\boldsymbol n),
    \end{equation}
where $\mc{P}: \bV_h \to \md{T}_n^0$, is the $L^2-$projector onto the space of mean zero trigonometric polynomials. This suggests the following iterative strategy: given an initial $g_0\in H^{1/2}(\Gamma)$, it can be used as Dirichlet datum for the HDG solver which will  produce a solution pair $(\boldsymbol q_h,u_h)$ to the interior problem \eqref{eq:interior-problem}. The flux $\boldsymbol q_h$ obtained in this fashion can then be transformed, using \eqref{lambda-qh}, into the Neumann datum for the exterior problem \eqref{eq:exterior_problem} and the process continues until the succesive solutions have stabilized. Note that $\boldsymbol n$ is the normal vector of the artificial boundary $\Gamma$ (rather than the normal vector of the computational boundary $\Gamma_h$, which is denoted by $\boldsymbol{n}_h$) hence, the approximation obtained on the computational domain $\Omega_h$ must be first extrapolated to $\Gamma$ and then projected onto $\mathbb T_n^0$ . 

This algorithm amounts to a Schur complement strategy where the Dirichlet-to-Neumann map (DtN) for the interior problem is approximated via HDG, and the Neumann-to-Dirichlet mapping (NtD) for the exterior problem is approximated via spectral BEM. As we have shown in the previous sections, both of these problems are uniquely and stably solvable, therefore, it remains to show that the iterated composition of these mappings will converge, and that the limits will in fact be the discrete Dirichlet and Neumann traces  over $\Gamma$ of the solution to \eqref{eq:ext-diff-problem}.

To explain the procedure at the continuous level we start by fixing $f\in L^2(\Omega)$ and $u_0\in H^{1/2}(\Gamma)$, and defining the mapping $T: H^{1/2}(\Gamma) \rightarrow H^{1/2}(\Gamma)$ that associates to $g\in H^{1/2}(\Gamma)$ the function $Tg\in H^{1/2}(\Gamma)$ given by the following two-step process:
\begin{subequations}\label{eq:IterativeProcess}
\begin{align}
\intertext{\textbf{Step 1:} Solve the interior Dirichlet boundary value problem}
\label{eq:Step1}
\left.\begin{array}{cl}
\phantom{\nabla\cdot\boldsymbol q\,}\nabla\cdot\boldsymbol q =\, f & \qquad \text{ in } \Omega\\
\phantom{n}\boldsymbol q + \boldsymbol\kappa\nabla u =\, 0 &\qquad \text{ in } \Omega\\
\phantom{\nabla\cdot\boldsymbol q nnn}u =\, g &\qquad \text{ on } \Gamma \\
\phantom{\nabla\cdot\boldsymbol q nnn\,\,}u =\, u_0 &\qquad \text{ on } \Gamma_0\\
\end{array}\right. &
\intertext{\textbf{Step 2:} Solve the boundary integral equation }
\label{eq:Step2}
\left.\begin{array}{cl}
& \\
\left(\tfrac{1}{2}-\mathcal K\right)Tg =\, -\mathcal V(\boldsymbol q\cdot\boldsymbol n) & \qquad \text{ on } \Gamma \\
&
\end{array}\right. &
\end{align}
\end{subequations}

We can then summarize the algorithm as, staring from an initial boundary datum $g_0\in H^{1/2}(\Gamma)$, generating a sequence  of updates by $g^{n+1} = Tg^n$. 
 The iterative process is continued until the relative change between consecutive iterations falls below a prescribed tolerance. An essentially equivalent idea (where the problems in the two domains are dealt with in PDE form) has been known to the domain decomposition community for a while; it can be traced back at least to \cite{BjWi1986}, where it was used as preconditing step  within a Schur complement algorithm to determine the Dirichlet trace along $\Gamma$ of the solution. The convergence of this straightforward idea depends on specific properties of the domains and can not be ensured in general, however a relaxed version of the method was proposed in \cite{FuQuZa1988,MaQu1988} and proven to be convergent in \cite{MaQu1989}.

What we will show in this section is that, as the distance between $\Gamma$ and $\Gamma_h$ tends to zero, the convergence of this procedure is not affected by the introduction of boundary integral equation and the transfer of boundary information between the non-touching grids.

\subsection{Continuous problem}

\paragraph{Fixed point operator and relaxation.}\label{sec:fixed-relax}

We start by introducing the space of admissible Neumann traces for the exterior problem at the continuous level
\[
X:=\left \{\mu \in H^{-1/2}(\Gamma): \int_\Gamma \mu = 0\right\}.
\]
The Dirichlet to Neumann mapping for the interior problem is then defined as
\begin{align}
\nonumber
\SDtN: H^{1/2}(\Gamma) &\longrightarrow\, X \\
\label{S1-operator}
g &\longmapsto\, (\boldsymbol{q}^g \cdot \boldsymbol{n})|_{\Gamma},
\end{align}
where $\boldsymbol{q}^g$ is the first component of  $(\boldsymbol{q}^g,u^g)$, the unique solution of \eqref{scheme-augmented-continuous} having $g$ and $u_0$ as Dirichlet boundary data on $\Gamma$ and $\Gamma_0$, respectively, and source term $f$. We can deduce a stability estimate for $S_1$ as follows. From the trace inequality for functions in $\bH(\bdiv;\Omega)$, and the continuous dependence \eqref{ineq:ContDependence}, we know that there exists a positive constant $C_{\SDtN}$ such that
\begin{equation*}
\|\SDtN g\|_{-1/2,\Gamma} \leq C_{\SDtN}\overline{\boldsymbol\kappa}^{1/2} \left( \|f\|_{0,\Omega} + \|g\|_{1/2,\Gamma} +    \|u_0\|_{1/2,\Gamma}\right) \qquad \forall g \in H^{1/2}(\Gamma).
\end{equation*}

Similarly, we can define the Neumann to Dirichlet map for the exterior problem as
\begin{align*}
\SNtD: X \longrightarrow&\, H^{1/2}(\Gamma)  \\
\lambda \longmapsto&\, g^\lambda |_{\Gamma},
\end{align*}
where $g^\lambda$ is the unique solution of \eqref{eq:IntegralProblem} having $\lambda$ as  Neumann boundary data on $\Gamma$. Moreover, from the continuous dependence \eqref{ineq:ContDependenceBIE}, there exists a positive constant $C_{\SNtD}$ such that
\begin{equation*}
\|\SNtD \lambda\|_{1/2,\Gamma} \leq C_{\SNtD} \|\lambda\|_{-1/2,\Gamma} \qquad \forall \lambda \in X.
\end{equation*}

The iterative procedure consists on the alternated application of these mappings, and is thus described by the repeated application of the operator
\begin{align*}
\T: H^{1/2}(\Gamma) \longrightarrow&\,H^{1/2}(\Gamma)  \\
g \longmapsto&\, \T g:= (\SNtD \circ \SDtN) g,
\end{align*}
which, by the arguments given above,  satisfies the stability estimate
\begin{equation*}
\|\T g\|_{1/2,\Gamma} \leq  C_{\SNtD} C_{\SDtN}\overline{\boldsymbol\kappa}^{1/2}\left( \|f\|_{0,\Omega} + \|g\|_{1/2,\Gamma} +    \|u_0\|_{1/2,\Gamma}\right) \qquad \forall g \in H^{1/2}(\Gamma).
\end{equation*}

As mentioned earlier, the simple iterative process described in previous section is not convergent in general. However this drawback can be overcome by  the introduction of an additional relaxation step and a relaxation parameter $\omega\in(0,1)$, resulting in 
\begin{subequations}
\label{eq:RelaxedIterativeProcess}
\begin{align}
\intertext{\textbf{Step 1:} Solve the interior Dirichlet boundary value problem}
\left.\begin{array}{cl}
\phantom{\nabla\cdot\boldsymbol q\,}\nabla\cdot\boldsymbol q^{n} =\, f & \qquad \text{ in } \Omega\\
\boldsymbol q^{n} + \boldsymbol\kappa\nabla u^{n} =\, 0 &\qquad \text{ in } \Omega\\
\phantom{\nabla\cdot\boldsymbol q nnnnn}u^{n} =\, g^{n-1} &\qquad \text{ on } \Gamma \\
\phantom{\nabla\cdot\boldsymbol q nnn\,\,}u^{n} =\, u_0 &\qquad \text{ on } \Gamma_0\\
\end{array}\right. &
\intertext{\textbf{Step 2:} Solve the boundary integral equation}
\left.\begin{array}{cl}
& \\
\left(\tfrac{1}{2}-\mathcal K\right)\tilde g =\, -\mathcal V(\boldsymbol q^{n}\cdot\boldsymbol n) & \qquad \text{ on } \Gamma \\
&
\end{array}\right. & 
\intertext{\textbf{Step 3:} Update the Dirichlet trace}
g^{n} =\, \omega \tilde g + (1-\omega)g^{n-1}. \qquad \qquad \qquad & 
\end{align}
\end{subequations}

We will denote the operator mapping a trace $g$ to the update defined by the relaxed process described above by $\Tomega: H^{1/2}(\Gamma) \longrightarrow\,H^{1/2}(\Gamma)$, and note that $\Tomega = \omega T + (1-\omega) I$, where $I$ is the identity operator. The following simple observation will be key in our analysis.

\begin{lem}\label{lem:observation}
Assume that $g\in H^{1/2}(\Gamma)$ is a fixed point of the relaxed operator $T_\omega$ (i.e. $T_\omega =g$). Then $g$ is also a fixed point of the unrelaxed operator $T$.
\end{lem}
\begin{proof}
If $g$ is a fixed point of $T_\omega$ it follows that $g = T_\omega g = \omega Tg + (1-\omega)g$. A simple calculation shows that this implies that $Tg = g$.
\end{proof}

\paragraph{Contraction property of $T_\omega$.}

We will now show that the relaxed mapping is indeed a contraction and therefore, by the observation above, the operator $T$ has indeed a fixed point. To do so, we will adapt the ideas applied by Marini and Quarteroni in \cite{MaQu1989}, where they dealt with a primal formulation involving  only PDE formulations in the two subdomains.

We are interested in showing that the repeated application of the operator $T_\omega$ is a contraction. With this in mind, we observe that the difference between successive applications $T_\omega^ng$ and $T_\omega^{n+1}g$ will be associated with the solution to an interior boundary value problem with source term $f=0$ and boundary condition $u_0=0$ on $\Gamma_0$. With these two ideas in mind we associate to every $\xi \in H^{1/2}(\Gamma)$ the function $\boldsymbol q^\xi \in \bH(\bdiv;\Omega)$ satisfying the interior boundary value problem
\begin{equation}\label{eq:AuxProb}
\left. \begin{array}{rll}
(\boldsymbol \kappa^{-1}\boldsymbol q^\xi,\boldsymbol v)_\Omega - ( u^\xi,\nabla\cdot\boldsymbol v)_\Omega = & -\langle\boldsymbol v\cdot\boldsymbol n,\xi\rangle_\Gamma & \qquad \forall \boldsymbol v \in H(\text{div},\Omega), \\
(v,\nabla\cdot \boldsymbol q^\xi)_\Omega =& 0 & \qquad \forall v \in L^2(\Omega).
\end{array}\right\}
\end{equation}
The problem above is a particular instance of \eqref{scheme-augmented-continuous}, which has been shown to be uniquely solvable. Recalling that $\partial\Omega = \Gamma \cup \Gamma_0$, the first equation implies that the trace of $u^{\xi}$ over $\Gamma_0$ vanishes. With this in mind it is easy to check that $\boldsymbol q^\xi = \boldsymbol 0$ if and only if $\phi = 0$ from which it follows that $\boldsymbol q^\xi = \boldsymbol q^\psi$ implies $\xi = \psi$. We will use this mapping and the fact that $\boldsymbol\kappa$ is symmetric and positive definite positive to define the inner product over $ H^{1/2}(\Gamma)$
\begin{equation}\label{eq:AuxInner}
\inner{\xi}{\psi}:= (\boldsymbol\kappa^{-1}\boldsymbol q^\xi, \boldsymbol q^\psi)_{\Omega} = (\boldsymbol\kappa^{-1}\boldsymbol q^\psi, \boldsymbol q^\xi)_{\Omega} \qquad \forall \xi, \psi \in \widetilde H^{1/2}(\Gamma).
\end{equation}
This induces a norm over $H^{1/2}(\Gamma)$ given by
\begin{equation*}
\triple{\xi} := \inner{\xi}{\xi}^{1/2}.
\end{equation*}
Moreover, from the definition of $\boldsymbol q^\phi$ and $\boldsymbol q^\psi$, it follows that
\begin{equation}
\label{eq:innerproductprop}
\inner{\xi}{\psi} =  -\langle \xi,\boldsymbol q^\psi\cdot\boldsymbol n\rangle_\Gamma =  -\langle \psi, \boldsymbol q^\xi\cdot\boldsymbol n \rangle_\Gamma.
\end{equation}

\begin{lem}\label{lem:equiv}
The following estimates hold for $g \in H^{1/2}(\Gamma)$
\begin{align}
\label{eq:TripleNormBound}
\triple{g}^2 \leq \,&  \frac{\overline{\boldsymbol\kappa}}{\underline{\boldsymbol \kappa}}C_{S_1}\|g\|^2_{1/2,\Gamma},\\
\label{eq:NegativeBound}
\inner{g}{Tg}\leq\,&-c\|Tg\|^2_{1/2,\Gamma},\\
\label{eq:gleqTg}
\triple{Tg} \leq\,& \frac{C_{S_1}\overline{\boldsymbol\kappa}}{c\underline{\boldsymbol\kappa}} \triple{g},\\
\label{eq:Tgleqg}
\triple{g} \leq\,& C_{PS}\sigma \|Tg\|_{1/2,\Gamma}.
\end{align}
\end{lem}
\begin{proof}
The first estimate follows readily from the definition of the inner product $\inner{\cdot}{\cdot}$ in \eqref{eq:AuxInner}, and the stability estimate for the interior problem
\[
\triple{g}^2 =\, \inner{g}{g} = \|\boldsymbol\kappa^{-1/2}\boldsymbol q^g\|^2_{\Omega}\leq \frac{C_{S_1}\overline{\boldsymbol\kappa}}{c\underline{\boldsymbol\kappa}}\|g\|^2_{1/2,\Gamma}.
\]
For \eqref{eq:NegativeBound} we start from \eqref{eq:innerproductprop} and make use of the fact that, by construction, $Tg$ satisfies the boundary integral equation \eqref{eq:Step2}, leading to   
\begin{equation}\label{eq:innerh-to-innerBEM}
\inner{g}{Tg}= \langle Tg, \boldsymbol q^g\cdot\boldsymbol n\rangle_\Gamma = -\left\langle Tg,\mathcal V ^{-1}\left(\tfrac{1}{2}-\mathcal K\right)Tg \right\rangle_\Gamma.
\end{equation}
Using now the representation $\mathcal V ^{-1}\left(\tfrac{1}{2}-\mathcal K\right) = \mathcal W + \left(\tfrac{1}{2}-\mathcal K^\prime\right)\mathcal V^{-1}\left(\tfrac{1}{2}-\mathcal K\right)$, it is possible to show \cite{SaSc2011} that there exists a positive constant $c$ such that
\begin{equation}
c\|g\|^2_{1/2,\Gamma}\leq\left\langle g,\mathcal V ^{-1}\left(\tfrac{1}{2}-\mathcal K\right)g \right\rangle_\Gamma\label{ineq:coerc}.
\end{equation}
Combining the last two expressions we arrive at \eqref{eq:NegativeBound}. The inequality \eqref{eq:gleqTg} follows readily from \eqref{eq:TripleNormBound} and \eqref{eq:NegativeBound} as follows
\[
\triple{Tg}^2 \leq C_{S_1}\frac{\overline{\boldsymbol \kappa}}{\underline{\boldsymbol \kappa}}\|Tg\|^2_{1/2,\Gamma}\leq -\frac{C_{S_1}\overline{\boldsymbol \kappa}}{c\underline{\boldsymbol \kappa}}\inner{g}{Tg}\leq \frac{C_{S_1}\overline{\boldsymbol \kappa}}{c\underline{\boldsymbol \kappa}}\triple{g}\triple{Tg}.
\]
Finally, we will use the fact that $\boldsymbol q^g$ and $g$ are linked by the interior problem \eqref{eq:AuxProb} as follows
\begin{alignat*}{10}
\triple{g}^2 = \inner{g}{g} = (\boldsymbol\kappa^{-1}\boldsymbol q^g, \boldsymbol q^g)_{\Omega} =\,& -\langle \boldsymbol{q}^g\cdot\boldsymbol n,g \rangle_\Gamma \qquad && (\text{By \eqref{eq:AuxProb} with } \boldsymbol v = \boldsymbol q^g)\\
=\,& \langle \mathcal V^{-1}\left(\tfrac{1}{2}-\mathcal K\right) Tg, g \rangle_\Gamma \qquad && (\text{By \eqref{eq:Step2}}) \\
\leq\,& C_{PS}\|Tg\|_{1/2,\Gamma}\|g\|_{1/2,\Gamma} && \\
\leq\,&  C_{PS}\sigma\|Tg\|_{1/2,\Gamma}\triple{g}, &&
\end{alignat*}
where in the last inequality we have appealed to an argument from \cite{LiMa1972,MaQu1989} pointing to the existence of a positive constant $\sigma$ such that 
\begin{align}\label{norm-equiv-sigma}
\|g\|_{1/2,\Gamma}\leq\sigma \triple{g},
\end{align}
and the constant $C_{PS}$ follows from the continuity of the Poincar\'e-Steklov operator $\mathcal V^{-1}\left(\tfrac{1}{2}-\mathcal K\right)$.
\end{proof}

Using the estimates from the previous lemma, we can now compute 
\begin{alignat}{6}
\nonumber
\triple{T_\omega g}^2 =\,& \omega^2\triple{Tg}^2 + (1-\omega)^2\triple{g}^2+ 2\omega(1-\omega)\inner{g}{Tg} &&\\
\nonumber
\leq\,& \omega^2\triple{Tg}^2 + (1-\omega)^2\triple{g}^2- 2\omega(1-\omega)c\|Tg\|_{1/2,\Gamma}^2 &&\qquad \text{(By \eqref{eq:NegativeBound})}\\
\nonumber
\leq\,& \left(\frac{\omega C_{S_1}\overline{\boldsymbol \kappa}}{c\underline{\boldsymbol\kappa}}\right)^2\triple{g}^2 + (1-\omega)^2\triple{g}^2- 2\omega(1-\omega)c\|Tg\|_{1/2,\Gamma}^2 &&\qquad \text{(By \eqref{eq:gleqTg})}\\
\nonumber
\leq\,& \left(\frac{\omega C_{S_1}\overline{\boldsymbol \kappa}}{c\underline{\boldsymbol\kappa}}\right)^2\triple{g}^2 + (1-\omega)^2\triple{g}^2- \frac{2\omega(1-\omega)c}{(\sigma C_{PS})^2}\triple{g}^2 &&\qquad \text{(By \eqref{eq:Tgleqg}}\\
\nonumber
=\,& \widehat C(\omega)\triple{g}^2, &&,
\end{alignat}
where we have defined
{\small
\[
\widehat C(\omega):=\left(\left(\frac{\omega C_{S_1}\overline{\boldsymbol \kappa}}{c\underline{\boldsymbol\kappa}}\right)^2 + (1-\omega)^2- \frac{2\omega(1-\omega)c}{(\sigma C_{PS})^2}\right)= \left(\left(\frac{ C_{S_1}\overline{\boldsymbol \kappa}}{c\underline{\boldsymbol\kappa}}\right)^2 +\frac{2c}{(\sigma C_{PS})^2}+1\right)\omega^2 -2\left(1+\frac{c}{(\sigma C_{PS})^2}\right)\omega + 1.
\]
}
We note that the quantity $\widehat C(\omega)$ is a continuous function of the relaxation parameter $\omega$ that attains its minimum value for
\[
\omega = \omega_m := \frac{1+\frac{c}{(\sigma C_{PS})^2}}{1+\frac{2c}{(\sigma C_{PS})^2} + \frac{C_{S_1}}{c\underline{\boldsymbol\kappa}}}\in(0,1).
\]
This implies that $\widehat C(\omega)$ is a decreasing function of $\omega$ within the interval $(-\infty,\omega_m)$. Therefore, since $\widehat C(0)=1$, we conclude that there exists $\omega^*>0$ such that for every $\omega\in(0,\omega^*)$ it holds that $0<\widehat C(\omega)<1$. Combining this argument with Lemma \ref{lem:observation}, we have thus proven the following

\begin{thm}
There exists $\omega^*>0$ such that, for any value of the relaxation parameter $\omega\in(0,\omega^*)$, the mapping $T_\omega$ is a contraction. As a consequence, the iterative procedure described by the problems \eqref{eq:RelaxedIterativeProcess} converges to the functions $\boldsymbol q, u, g$ satisfying problems \eqref{eq:IterativeProcess}.
\end{thm}

\subsection{Discrete problem}
We will follow the main ideas introduced for the analysis of the continuous counterpart, but we will have to adapt them to account for the additional challenges posed by the discretization and the transfer technique. 

\paragraph{Discrete fixed point operator and relaxation.}

In this section we construct the discrete counterpart of the operators defined in Section \ref{sec:fixed-relax}. To that end, we let  
\[
X_h:=\left\{\mu \in L^2(\Gamma): \forall e \in \mathcal{E}_h^\partial, \,\mu\vert_{\Gamma_e} = (E\boldsymbol{p}_h\cdot\boldsymbol{n})\vert_{\Gamma_e} \,\,{\rm with } \,\,\boldsymbol{p}_h\in  [\mathbb{P}_k(T_e)]^2 \,\,{\rm and}\,\, \displaystyle \int_{\Gamma} \mu = 0\right\},
\]

and define the discrete version of the operator $\SDtN$ (cf. \eqref{S1-operator}) as
\begin{align*}
\Sh: \mathbb{P}_k(\mathcal{E}_h^\partial )\longrightarrow&\, X_h\\
g_h \mapsto&\, \Sh g_h:= (E\boldsymbol{q}_h^{g} \cdot \boldsymbol{n})|_{\Gamma},
\end{align*}
where $\boldsymbol{q}_h^{g}$ is the first component of  $(\boldsymbol{q}_h^{g},u_h^g)$, the unique solution of \eqref{scheme-augmented} having $g_h$ and $u_0$ as Dirichlet boundary data on $\Gamma$ and $\Gamma_0$, resp., and source term $f$. Moreover, by \eqref{ estim HDG}, we have that
\begin{align*}
       \J(\bq_h^g,u_h^g) 
    &\leq C_{HDG}\, \left( \| f\|_{0,\Omega_h} + \|\kap^{1/2}h^{-1/2} u_0\|_{\Gamma_0}+\|\kap^{1/2}h^{-1/2}g_h\circ \phi\|_{\Gamma} \right).
\end{align*}
On the other hand, consider a mesh edge $e\in \mathcal{E}_h^\partial$ and recall the bijective mapping $\phi$, defined in \eqref{eq:BoundaryMapping}; we will denote the image of an edge $e\subset\Gamma_h$ under $\phi$ by $\Gamma_e := \phi(e)$. Now, by considering Lemma 4 in \cite{CaSo2021}, it is possible to deduce that there exists a non-negative constant $C_{\Gamma_e}$, independent of $h$, such that
\begin{equation}\label{ineq:traceGamma}
    \| E\boldsymbol{q}_h^{g} \cdot \boldsymbol{n} \|_{\Gamma_e} \leq C_{\Gamma_e} C_{ext}^e C_2^e h_e^{-1/2} \|\boldsymbol{q}_h^{g}\|_{T^e}.
\end{equation}

Therefore, the above two estimates imply that there exists $C_{\Sh}>0$, independent of $h$, such that
\begin{equation}\label{ineq:continuity-Sh}
    \| \Sh g_h \|_{0,\Gamma} \leq C_{\Sh} h^{-1/2}\left( \| f\|_{0,\Omega_h} + \|\kap^{1/2}h^{-1/2}u_0\|_{\Gamma_0}+\|\kap^{1/2}h^{-1/2}g_h\circ \phi\|_{\Gamma} \right).
\end{equation}

Similarly, the discrete version of the operator $\SNtD$ is given by
\begin{align*}
\Sn: \mathbb{T}_n^0\longrightarrow&\, \mathbb{T}_n^0\\
\lambda_n \mapsto&\, \Sn \lambda_n:= g_n^0,
\end{align*}
where $g_n^0$ is the unique solution of the equation \eqref{int-identity_app} with Neumann data $\lambda_n$, and satisfies
\begin{equation}\label{BEM-cont-dep}
\|\Sn \lambda_n\|_{1/2,\Gamma} =  \|g_n^0\|_{1/2,\Gamma}\leq C_{BEM}   \|\lambda_n\|_{-1/2,\Gamma}.
\end{equation}

We can now define the following discrete analogue to the operator $T$ from Section \ref{sec:fixed-relax} as 
\begin{align*}
\T^{h,n}:\mathbb{T}_n^0 \longrightarrow&\,\mathbb{T}_n^0  \\
g_n^0 \longmapsto&\, \T^{h,n} g_n^0:= \Sn \circ \mathbb{Q}_n^0\circ  \Sh\circ \Pi_h \circ (g_n^0 \circ \phi),
\end{align*}
where $\Pi_h$ and $\mathbb{Q}_n^0$ are the $L^2$-projections into $\mathbb{P}_k(\mathcal{E}_h^\partial)$ and $\mathbb{T}_n^0$, respectively.

\paragraph{ Contraction property of $\T^{h,n}$.}

We define the discrete version of \eqref{eq:AuxInner}. For $\varphi, \psi \in  \mathbb{T}_n^0$,
\begin{equation}\label{eq:AuxInnerh}
\innerh{\varphi}{\psi}:= \mc{A}_h(\bq_h^\varphi, \bq_h^{\psi}) + \mc{C}_h(u_h^\varphi,u_h^{\psi}) 
\end{equation}
where $(\boldsymbol q^\varphi_h, u_\varphi)$ and $(\boldsymbol q^\psi_h, u_\psi)$ are the solutions to \eqref{scheme-augmented} with source term $f=0$, $u_0=0$ on $\Gamma_0$ and boundary data over $\Gamma$ given by $\varphi$ and $\psi$ respectively. This is, in fact, an inner product on $\mathbb{T}_n^0$. In order to see that, first let us note that $\mc{C}_h$ is a semi-definite positive operator from $W_h\times W_h$ (cf. \eqref{Ch-semidefinite}).
Therefore, if $\innerh{\psi}{\psi} = 0$, then $\boldsymbol q^\psi_h = \boldsymbol 0$ and $\mc{C}(u^\psi_h,u^\psi_h) =0$. Moreover, by \eqref{Ch-semidefinite} we have that $u_h^\psi$ is single-valued and vanishes on the the boundary. Thus, considering all this information, from \eqref{Hdg discrete-2} we have that 
\[
\langle \tau\, u_h^\psi, w \rangle_{\partial \mc{T}_h} -  2 \langle \tau u_h^\psi,\prom{w}\rangle_{\mc{E}_h^{\circ}} =  - \langle \tau \,  \psi\circ\phi, w \rangle_{\Gamma} \qquad \forall w \in  W_h.
\]
Now, expressing the integral over $\partial \mathcal{T}_h$ in terms of summation over edges and recalling that $u_h^\psi = \prom{u_h^\psi }$ and  $u_h^\psi = 0$ on the boundary, we deduce that the right hand side of the expression above must vanish for all $w\in W_h$. In particular, taking $w=1$ it follows that
\[
0 =  - \langle \tau \,  \psi\circ\phi, 1\rangle_{\partial\Omega_{h}} = - \langle \tau \,  \psi, 1 \rangle_{\Gamma}.
\]
Therefore, since $\tau$ is positive and $\psi\in\mathbb{T}_n^0$, we must have $\psi=0$. This inner product induces the norm
$
\normh{\varphi}:=\innerh{\varphi}{\varphi}^{1/2}
$
and we notice that
\begin{equation}\label{triplenormh}
\normh{\varphi}^2=\triple{\varphi}^2+ \frac{1}{2}\|\tau^{-1/2}\jump{\boldsymbol{q}_h^{\varphi}}\|_{\mathcal{E}_h^o}^2 + \mathcal{C}_h(u_h^{\varphi},u_h^{\varphi})\geq \triple{\varphi}^2.
\end{equation}

We now establish the relationship between the discrete norm $\triple{\cdot}_h$, the continuous norms in $H^{1/2}(\Gamma)$ and $L^2(\Gamma)$.
\begin{lem}\label{gn:equiv1}
Let $g_n\in \mathbb{T}_n^0$. There hold
\begin{align}\label{Ltwptonormh}
\|g_n\|_{\Gamma}\leq\|g_n\|_{1/2,\Gamma}\leq \sigma \normh{g_n}.
\end{align}

\end{lem}
\begin{proof}
Let $g_n \in \mathbb{T}_n^0$. By employing \eqref{norm-equiv-sigma} we have that
$
\|g_n\|_{1/2,\Gamma}^2
\leq\sigma^2\triple{g_n}^2
\leq \sigma^2 \normh{g_n}^2
$,
where in the last inequality we made use of \eqref{triplenormh}. The second inequality follows by the characterization of the $H^{1/2}$-norm in terms of the Fourier coefficients of the function and, the fact that the parametrization of $\Gamma$ is smooth, and the fact that $g_n$ is a trigonometric polynomial (see, for instance, \cite{SaVa2002}).
\end{proof}

The following identity and the one in the subsequent corollary establish the connection between the inner product $\innerh{\cdot}{\cdot}$, defined through the interior problem, and the exterior problem. This will play a key role in deriving the discrete analogue of \eqref{eq:innerh-to-innerBEM}.
\begin{lem}
Let $\varphi,\psi \in \mathbb{T}^0_n$. There holds
\begin{align}\label{identity-inner-prod1}
\innerh{\varphi}{\psi}
=&
-\left\langle\varphi, \mathcal{V}^{-1}\mathbb{Q}_n^0\left(\frac{1}{2}-\mathcal{K}\right)S_n(\mathbb{Q}_n^0(\boldsymbol q^{\psi }_h  \circ \phi^{-1})) \right\rangle_{\Gamma}
-
  \langle(Id-\mathbb{Q}_n^0)(\mathcal{V}^{-1}\varphi,  \mathcal{V}\mathbb{Q}_n^0((\boldsymbol q^{\psi }_h  \circ \phi^{-1})\cdot \boldsymbol{n})\rangle_{\Gamma}\nonumber
\\
  & -\langle\varphi,  (\boldsymbol q^{\psi }_h  \circ \phi^{-1})\cdot ({\boldsymbol n_h}-\boldsymbol{n})\rangle_{\Gamma}
  +\langle \tau \,  \varphi\circ\phi, u_h^{\psi} \rangle_{\Gamma_h} -\mc{A}_T(\bq_h^\varphi, \bq_h^{\psi})+\mc{B}_T(\bq_h^\varphi, u^{\psi}).
\end{align}
\end{lem}
\begin{proof}

Let $\varphi,\psi \in \mathbb{T}^0_n$. By the definition of $\innerh{\varphi}{\psi}$ and the equations \eqref{scheme-augmented} satisfied by $(\boldsymbol q^\varphi_h, u_\varphi)$ and $(\boldsymbol q^\psi_h, u_\psi)$, it is possible to deduce the identity
\begin{align*}
\innerh{\varphi}{\psi} =&\mc{F}_{1,h}(\bq_h^\psi)-\mc{F}_{2,h}( u^{\psi}) -\mc{A}_T(\bq_h^\varphi, \bq_h^{\psi})+\mc{B}_T(\bq_h^\varphi, u^{\psi})\\
=& - \langle \varphi\circ\phi, \bq_h^\psi\cdot {\boldsymbol n_h} \rangle_{\Gamma_h}+\langle \tau \,  \varphi\circ\phi, u_h^{\psi} \rangle_{\Gamma_h}  -\mc{A}_T(\bq_h^\varphi, \bq_h^{\psi})+\mc{B}_T(\bq_h^\varphi, u^{\psi}).
\end{align*}

Now, since $\phi$ is a bijective mapping, we write the first term of the right hand side as follows:
\begin{align*}
 -\langle\varphi  \circ \phi,  \boldsymbol q^{\psi }_h \cdot {\boldsymbol n_h}\rangle_{\Gamma_h}  
 =& -\langle\varphi,  (\boldsymbol q^{\psi }_h  \circ \phi^{-1})\cdot {\boldsymbol n_h}\rangle_{\Gamma}\nonumber\\
  =&  -\langle\varphi,  (\boldsymbol q^{\psi }_h  \circ \phi^{-1})\cdot \boldsymbol{n}\rangle_{\Gamma}
   -\langle\varphi,  (\boldsymbol q^{\psi }_h  \circ \phi^{-1})\cdot ({\boldsymbol n_h}-\boldsymbol{n})\rangle_{\Gamma}\nonumber\\
  =&  -\langle\varphi,  \mathbb{Q}_n^0((\boldsymbol q^{\psi}_h  \circ \phi^{-1})\cdot \boldsymbol{n})\rangle_{\Gamma}
   -\langle\varphi,  (\boldsymbol q^{\psi }_h  \circ \phi^{-1})\cdot ({\boldsymbol n_h}-\boldsymbol{n})\rangle_{\Gamma},
\end{align*}
where we have added and subtracted ${\boldsymbol n_h}$ and used the fact that $\varphi \in \mathbb{T}_n^0$ in the last step.

We now conveniently rewrite the first term on the right hand. More precisely, since $\mathcal{V}$ is invertible and self-adjoint,
\begin{align*}
 \langle\varphi,  \mathbb{Q}_n^0((\boldsymbol q^{\psi }_h  \circ \phi^{-1})\cdot \boldsymbol{n})\rangle_{\Gamma}
 =& \langle\mathcal{V}^{-1}\varphi,  \mathcal{V}\mathbb{Q}_n^0((\boldsymbol q^{\psi }_h  \circ \phi^{-1})\cdot \boldsymbol{n})\rangle_{\Gamma}\nonumber\\
  =& \langle\mathbb{Q}_n^0(\mathcal{V}^{-1}\varphi),  \mathcal{V}\mathbb{Q}_n^0((\boldsymbol q^{\psi }_h  \circ \phi^{-1})\cdot \boldsymbol{n})\rangle_{\Gamma}\nonumber\\
  &+
  \langle(Id-\mathbb{Q}_n^0)(\mathcal{V}^{-1}\varphi),  \mathcal{V}\mathbb{Q}_n^0((\boldsymbol q^{\psi }_h  \circ \phi^{-1})\cdot \boldsymbol{n})\rangle_{\Gamma},
\end{align*}
where we added and subtracted $\mathbb{Q}_n^0(\mathcal{V}^{-1}\varphi)$. 

Then, taking $ \mathbb{Q}_n^0(\mathcal{V}^{-1}\varphi)$ as a test function in \eqref{int-identity_app} Neumann data $\lambda:= \mathbb{Q}_n^0(\boldsymbol q^{\varphi }_h  \circ \phi^{-1})$ and unique solution $g^\lambda:=S_n(\mathbb{Q}_n^0(\boldsymbol q^{\psi }_h  \circ \phi^{-1}))$, we have that
\begin{align*}
 \langle\varphi,  \mathbb{Q}_n^0((\boldsymbol q^{\psi }_h  \circ \phi^{-1})\cdot \boldsymbol{n})\rangle_{\Gamma}
  =& \left\langle\mathbb{Q}_n^0(\mathcal{V}^{-1}\varphi), \left(\frac{1}{2}-\mathcal{K}\right)g^\lambda \right\rangle_{\Gamma}+
  \langle(Id-\mathbb{Q}_n^0)(\mathcal{V}^{-1}\varphi,  \mathcal{V}\mathbb{Q}_n^0((\boldsymbol q^{\psi }_h  \circ \phi^{-1})\cdot \boldsymbol{n})\rangle_{\Gamma}\\
    =& \left\langle\varphi, \mathcal{V}^{-1}\mathbb{Q}_n^0\left(\frac{1}{2}-\mathcal{K}\right)g^\lambda \right\rangle_{\Gamma}+
  \langle(Id-\mathbb{Q}_n^0)(\mathcal{V}^{-1}\psi,  \mathcal{V}\mathbb{Q}_n^0((\boldsymbol q^{\psi }_h  \circ \phi^{-1})\cdot \boldsymbol{n})\rangle_{\Gamma},
\end{align*}

Gathering all the above identities, we obtain \eqref{identity-inner-prod1}.
\end{proof}

In the particular case of a circular interface $\Gamma$, the integral operators applied to trigonometric polynomials are also trigonometric polynomials. Therefore, we have the following identity.

\begin{crl}\label{cor:3}
Let us suppose that $\Gamma$ is a circular interface. For $\varphi,\psi \in \mathbb{T}^0_n$, there holds
\begin{align}\label{identity-inner-prod3}
\innerh{\varphi}{\psi}
=&
-\left\langle\varphi, \mathcal{V}^{-1}\left(\frac{1}{2}-\mathcal{K}\right)S_n(\mathbb{Q}_n^0(\boldsymbol q^{\psi }_h  \circ \phi^{-1})) \right\rangle_{\Gamma}
\nonumber
\\
  & -\langle\varphi,  (\boldsymbol q^{\psi }_h  \circ \phi^{-1})\cdot ({\boldsymbol n_h}-\boldsymbol{n})\rangle_{\Gamma}
  +\langle \tau \,  \varphi\circ\phi, u_h^{\psi} \rangle_{\Gamma_h} -\mc{A}_T(\bq_h^\varphi, \bq_h^{\psi})+\mc{B}_T(\bq_h^\varphi, u^{\psi}).
\end{align}

\end{crl}

We recall that the interface $\Gamma$ has been introduced artificially and its shape can be chosen to facilitate computations. In particular, all the boundary integrals can be explicitly computed in the case of a circular interface. This actually the case of the numerical examples reported in \cite{CoSaSo2012}. From now on, for the sake of simplicity of the exposition, we will consider $\Gamma$ is a circular interface.

The next lemma provides a discrete version of the inequalities presented in Lemma \ref{lem:equiv}. 
To that end, let us first notice that the solution $u$ of \eqref{scheme-augmented-continuous} is actually in $H^1(\Omega)$. In addition, if we assume that $\boldsymbol{q}\in \mbf{H}^1(\Omega)$, we have the following stability estimate
\begin{equation}\label{ineq:ContDependence-extra}
\|\boldsymbol{q}\|_{1,\Omega} + \|u\|_{1,\Omega}
\leq C_{\text{stab}}\overline{\boldsymbol \kappa}^{1/2}\left( \|f\|_{0,\Omega} + \|g\|_{1/2,\Gamma} +    \|u_0\|_{1/2,\Gamma_0}\right).
\end{equation}
\begin{lem}\label{lem:equiv-discrete}
Let $g_n\in \mathbb{T}^0_n$ and assume \eqref{ineq:ContDependence-extra} holds true. We have that
\begin{alignat}{6}\label{new-estimate-gnh-to-gnhalf}
\normh{g_n}^2
 \leq& 
C_0(\tau)\|g_n\|_{1/2,\Gamma}^2,
\end{alignat}
where
\begin{align*}
C_0(\tau):=C (\underline{\kap}^{-1/2}+\underline{\tau}^{-1/2}+1+(\overline{\tau}^{1/2}+1)\overline{\tau}^{1/2})(C_{\mathrm{stab}}\overline{\boldsymbol \kappa}^{1/2} +1) + \overline{\tau}    
\end{align*}
and
\begin{align} \label{psi-phi-aux2}
\innerh{g_n}{\T^{h,n}g_n }
 \leq& -c\|\T^{h,n}g_n\|_{1/2,\Gamma}^2
 +C_1(h,\tau)\normh{g_n}^2,
\end{align}
with
\begin{align*}
C_1(h,\tau) : =& C\bigg(R_h \underline{\kap}^{-1}\overline{\kap}
 +\underline{\kap}^{-1}\overline{\kap}^{1/2}R_h h^{1/2} \overline{\tau}^{1/2}+\sigma h^{-1/2}\overline{\kap}^{1/2}\|({\boldsymbol n_h}-\boldsymbol{n})\|_{\infty,\Gamma} \bigg).
\end{align*}
Moreover,
\begin{alignat}{6}\label{Tgn1}
\normh{\T^{h,n}g_n}^2
\leq& C_0(\tau)c^{-1}\left(C_0(\tau)c^{-1}+ C_1(h,\tau)\right)\normh{g_n}^2
\end{alignat}
and
\begin{align}\label{Tgn2}
\normh{g_n}^2
    \leq&C_{PS}^2 \sigma^2 \normh{\T^{h,n} g_n}^2
   + C_1(h,\tau)\sigma^2
 \normh{g_n}^2.
\end{align}
\end{lem}
\begin{proof}
To prove \eqref{ineq:ContDependence-extra} we start by using the definition of the norm $\triple{\cdot}_h$ to compute
\begin{alignat*}{6}
\normh{g_n}^2 =&\,
 \|\kap^{-1/2} \bq_h^{g_n}\|_{\Omega_h}^2 + \frac{1}{2} \| \tau^{-1/2} \jump{\bq_h^{g_n}}\|_{\mc{E}_h^{\circ}}^2+ 
\|\tau^{1/2} (u_h^{g_n} - \prom{u_h^{g_n}})\|_{\partial \mc{T}_h\setminus \partial \Omega_h}^2
+ \|\tau^{1/2} u_h^{g_n}\|_{ \partial \Omega_h}^2.\\
\intertext{However, since $u^{g_n}\in H^1(\Omega)$ and $\boldsymbol q^{g_n}\in H(\text{div},\Omega)$ it follows that}
\normh{g_n}^2 \leq&\,
 \|\kap^{-1/2} (\bq_h^{g_n}-\bq^{g_n})\|_{\Omega_h}^2 + \frac{1}{2} \| \tau^{-1/2} \jump{\bq_h^{g_n}-\bq^{g_n}}\|_{\mc{E}_h^{\circ}}+ 
\|\tau^{1/2} ((u_h^{g_n}-u^{g_n}) - \prom{u_h^{g_n}-u^{g_n}})\|_{\partial \mc{T}_h\setminus \partial \Omega_h}^2\\
&\,
+ \|\tau^{1/2} (u_h^{g_n}-u^{g_n})\|_{ \partial \Omega_h}^2
+\|\kap^{-1/2} \bq^{g_n}\|_{\Omega_h}^2
+ \|\tau^{1/2}u^{g_n}\|_{ \partial \Omega_h}^2\\
=&\,
\J(\bq-\bq_h,u-u_h)
+\|\kap^{-1/2} \bq^{g_n}\|_{\Omega_h}^2
+ \|\tau^{1/2}u^{g_n}\|_{ \partial \Omega_h}^2 \quad\qquad\text{(By the definition \eqref{def:J})}\\
=&\,
\J(\bq-\bq_h,u-u_h)
+\|\kap^{-1/2} \bq^{g_n}\|_{\Omega_h}^2
+ \|\tau^{1/2}u^{g_n}\|_{ \partial \Omega_h}^2 \quad\qquad\text{(By  \eqref{ineq:ContDependence})}\\
\leq&\,
C (\underline{\kap}^{-1/2}+\underline{\tau}^{-1/2}h^{1/2}+1) |\bq|_{1,\Omega}
+(\overline{\tau}^{1/2}+1)\overline{\tau}^{1/2} |u|_{1,\Omega} \!\qquad\text{(By  \eqref{ineq:Jqu})}\\
&+\,\|\kap^{-1/2} \bq^{g_n}\|_{\Omega_h}^2 
+ \|\tau^{1/2}u^{g_n}\|_{ \partial \Omega_h}^2 \\
\leq& C (\underline{\kap}^{-1/2}+\underline{\tau}^{-1/2}h^{1/2}+1+(\overline{\tau}^{1/2}+1)\overline{\tau}^{1/2})(C_{\text{stab}}\overline{\boldsymbol \kappa}^{1/2} +1)\|g_n\|_{1/2,\Gamma}^2
+ \overline{\tau}\|g_n\|_{\Gamma}^2 \qquad \text{(By \eqref{ineq:ContDependence-extra})},
\end{alignat*}
which implies  \eqref{new-estimate-gnh-to-gnhalf}.

Now, let $\varphi,\psi \in \mathbb{T}_n^0$. By the previous Corollary \ref{cor:3}, the Cauchy-Schwarz inequality and the continuity properties of the operators $\mc{A}_T$ and $\mc{B}_T$ (cf. \eqref{boundAT} and \eqref{boundBT}), and denoting by $C$ a generic positive constant independent of the discretization parameters, we can deduce that
\begin{align*}
\innerh{\varphi}{\psi}
 \leq& -\left\langle\psi, \mathcal{V}^{-1}\left(\frac{1}{2}-\mathcal{K}\right)S_n(\mathbb{Q}_n^0(\boldsymbol q^{\varphi }_h  \circ \phi^{-1})) \right\rangle_{\Gamma}
 -\langle\psi,  (\boldsymbol q^{\varphi }_h  \circ \phi^{-1})\cdot ({\boldsymbol n_h}-\boldsymbol{n})\rangle_{\Gamma}\\
 &+CR_h \underline{\kap}^{-1}\overline{\kap}\|\kap^{-1/2}\bq_h^{\varphi}\|_{\Omega_h}
 \|\kap^{-1/2}\bq_h^{\psi }\|_{\Omega_h}\\
 &+C\underline{\kap}^{-1}\overline{\kap}^{1/2}R_h h^{1/2} \overline{\tau}^{1/2}
\left(\frac{1}{2}\|\kap^{-1/2}\bq_h^{\varphi}\|_{\Omega_h}^2 +\frac{1}{2}\|\tau^{1/2}u_h^{\psi}\|_{\partial\Omega_h}^2\right)
\\
 \leq& -\left\langle\psi, \mathcal{V}^{-1}\left(\frac{1}{2}-\mathcal{K}\right)S_n(\mathbb{Q}_n^0(\boldsymbol q^{\varphi }_h  \circ \phi^{-1})) \right\rangle_{\Gamma}
 -\langle\psi,  (\boldsymbol q^{\varphi }_h  \circ \phi^{-1})\cdot ({\boldsymbol n_h}-\boldsymbol{n})\rangle_{\Gamma}\\
 &+\frac{1}{2}C\bigg(R_h \underline{\kap}^{-1}\overline{\kap}
 +\underline{\kap}^{-1}\overline{\kap}^{1/2}R_h h^{1/2} \overline{\tau}^{1/2}\bigg)\bigg(
\normh{\varphi}^2+\normh{\psi}^2\bigg).
\end{align*}

For the second term on the right hand side we have that
\begin{align*}
 -\langle\psi,  (\boldsymbol q^{\varphi }_h  \circ \phi^{-1})\cdot ({\boldsymbol n_h}-\boldsymbol{n})\rangle_{\Gamma}
 \leq &
 \|\psi\|_{\Gamma} \|\boldsymbol q^{\varphi }_h  \circ \phi^{-1}\|_{\Gamma}\| ({\boldsymbol n_h}-\boldsymbol{n})\|_{\infty,\Gamma}\\
\leq &\|\psi\|_{\Gamma} h^{-1/2}\overline{\kap}^{1/2}\|\kap^{-1/2}\boldsymbol{q}_h^{\varphi}\|_{\Omega_h} \|({\boldsymbol n_h}-\boldsymbol{n})\|_{\infty,\Gamma}
\\
\leq & \sigma h^{-1/2}\overline{\kap}^{1/2}  \|({\boldsymbol n_h}-\boldsymbol{n})\|_{\infty,\Gamma} \normh{\varphi} \normh{\psi},
\end{align*}
where in the last inequality we employed \eqref{Ltwptonormh} and the definition of $\normh{\cdot}$.
Hence,
\begin{align}\label{psi-phi-aux}
\innerh{\varphi}{\psi}
  \leq& -\left\langle\psi, \mathcal{V}^{-1}\left(\frac{1}{2}-\mathcal{K}\right)S_n(\mathbb{Q}_n^0(\boldsymbol q^{\varphi }_h  \circ \phi^{-1})) \right\rangle_{\Gamma}
\nonumber\\
 &+\frac{1}{2}C\bigg(R_h \underline{\kap}^{-1}\overline{\kap}
 +\underline{\kap}^{-1}\overline{\kap}^{1/2}R_h h^{1/2} \overline{\tau}^{1/2}+\sigma h^{-1/2}\overline{\kap}^{1/2}\|({\boldsymbol n_h}-\boldsymbol{n})\|_{\infty,\Gamma} \bigg)\bigg(
\normh{\varphi}^2+\normh{\psi}^2\bigg).
\end{align}

Now, by setting $\varphi=g_n$ and $\psi=\T^{h,n}g_n=S_n(\mathbb{Q}_n^0(\boldsymbol q^{g_n }_h  \circ \phi^{-1}))$ and recalling that $\innerh{\cdot}{\cdot}$ is symmetric, \eqref{psi-phi-aux} implies \eqref{psi-phi-aux2}.

On the other hand, \eqref{psi-phi-aux2} implies
\begin{alignat}{6}\label{aux10}
\|\T^{h,n}g_n\|_{1/2,\Gamma}^2 
\leq&
 c^{-1}\normh{g_n^0}\normh{\T^{h,n}g_n} +c^{-1} C_1(h,\tau) \normh{g_n}^2.
\end{alignat}
Then, by \eqref{new-estimate-gnh-to-gnhalf} and Young's inequality we obtain
\begin{alignat*}{6}
\normh{\T^{h,n}g_n}^2
\leq& C_0(\tau)\|\T^{h,n}g_n\|_{1/2,\Gamma}^2
\leq C_0(\tau)c^{-2} \normh{g_n}^2 +\normh{\T^{h,n}g_n}^2 
+c^{-1} C_1(h,\tau)C_0(\tau)\normh{g_n}^2
\end{alignat*}
and \eqref{Tgn1} follows.

Finally, taking $g_n=\psi=\varphi$ in \eqref{psi-phi-aux}, the definition of $C_1(h,\tau)$ and \eqref{Ltwptonormh}, we obtain
\begin{align*}
\normh{g_n}^2
 \leq& -\left\langle g_n, \mathcal{V}^{-1}\left(\frac{1}{2}-\mathcal{K}\right)S_n(\mathbb{Q}_n^0(\boldsymbol q^{g_n }_h  \circ \phi^{-1})) \right\rangle_{\Gamma}
 +C_1(h,\tau)
 \|g_n\|_{\Gamma}^2\\
 =&-\left\langle g_n, \mathcal{V}^{-1}\left(\frac{1}{2}-\mathcal{K}\right)\T^{h,n}g_n \right\rangle_{\Gamma}
 +C_1(h,\tau)
 \|g_n\|_{\Gamma}^2\\
 \leq&C_{PS} \sigma\|g_n\|_{1/2,\Gamma} \normh{T^{h,n} g_n}
 +C_1(h,\tau)\sigma^2
 \normh{g_n}^2\\
  \leq&C_{PS}\sigma^2\normh{g_n} \normh{\T^{h,n}}+
C_1(h,\tau)\sigma^2
 \normh{g_n}^2\\
   \leq& \frac{1}{2} \normh{ g_n }^2 +\frac{1}{2} C_{PS}^2 \sigma^2 \normh{\T^{h,n} g_n}^2
   +
 C_1(h,\tau)\sigma^2
 \normh{g_n}^2,
\end{align*}
which implies \eqref{Tgn2}.

\end{proof}

Similarly to the case of the operator $\Tomega$, we define the operator
\begin{align*}
\Tomega^{h,n}:\mathbb{T}_n^0 \longrightarrow&\,\mathbb{T}_n^0 \\
g_n^0 \mapsto&\, \Tomega^{h,n} g_n^0:= \omega \T^{h,n} g_n^0 + (1-\omega)g_n^0.
\end{align*}

We can now use the previous lemmas to prove the main result of this communication, namely the convergence of the iterative procedure.

\begin{thm}
If the mesh parameter $h$ is small enough, it is possible to find values of the relaxation parameter $\omega$ in the interval $(0,1)$  for which the discrete operator $\Tomega^{h,n}$ is a contraction. Therefore, the iterative procedure \eqref{eq:RelaxedIterativeProcess} converges.
\end{thm}

\begin{proof}

Let $g_n\in \mathbb{T}_n^0$. By employing the estimates in Lemma \ref{lem:equiv-discrete}
\begin{align*}
 \normh{\Tomega^{h,n} g_n}^2 =&\,
  \omega^2\normh{\T^{h,n} g_n}^2 
  +(1-\omega)^2 \normh{ g_n}^2
  +2\,\omega(1-\omega)\,\innerh{g_n}{\T^{h,n}g_n }\nonumber\\
  \leq&\,
   \omega^2 C_0(\tau)c^{-1}\left(C_0(\tau)c^{-1}+ C_1(h,\tau)\right)\normh{g_n}^2
  +(1-\omega)^2 \normh{ g_n}^2\\
  &\,-2c\,\omega(1-\omega)\,\|\T^{h,n}g_n\|_{1/2,\Gamma}^2
  + 2\omega(1-\omega) C_1(h,\tau)\normh{g_n}^2\\
    \leq&\,
   \omega^2 C_0(\tau)c^{-1}\left(C_0(\tau)c^{-1}+ C_1(h,\tau)\right)\normh{g_n}^2
  +(1-\omega)^2 \normh{ g_n}^2\\
  &\,-2c\,\omega(1-\omega)\,
C_0(\tau) \normh{\T^{h,n}g_n}^2
  + 2\omega(1-\omega) C_1(h,\tau)\normh{g_n}^2.
  \end{align*}
where in the last inequality we made use of \eqref{new-estimate-gnh-to-gnhalf}. Then, by \eqref{Tgn2},
\begin{align*}
 \normh{\Tomega^{h,n} g_n}^2
    \leq&\,
   \omega^2 C_0(\tau)c^{-1}\left(C_0(\tau)c^{-1}+ C_1(h,\tau)\right)\normh{g_n}^2
  \\
  &\,+(1-\omega)^2 \normh{ g_n}^2 -2c\,\omega(1-\omega)\,
C_0(\tau) C_{PS}^{-2} \sigma^{-2}\normh{g_n}^2\\ 
&\,+2c\,\omega(1-\omega)\,\sigma^{-2}
C_0(\tau)C_{PS}^{-2}  C_1(h,\tau)
 \normh{g_n}^2
  + 2\omega(1-\omega) C_1(h,\tau)\normh{g_n}^2\\
  = &\, \widehat C^{h,n}(\omega)  \normh{g_n}^2,
  \end{align*}

where
\begin{align*}
\widehat C^{h,n}(\omega) : = &\,
   \omega^2 C_0(\tau)c^{-1}\left(C_0(\tau)c^{-1}+ C_1(h,\tau)\right)
  +(1-\omega)^2 
  -2c\,\omega(1-\omega)\,
C_0(\tau) C_{PS}^{-2} \sigma^{-2}\\ 
&+2c\,\omega(1-\omega)\,
C_0(\tau)C_{PS}^{-2} \sigma^{-2} C_1(h,\tau)
  + 2\omega(1-\omega) C_1(h,\tau).
\end{align*}
  
Analogously to the analysis of the continuous operator, we observe that $\widehat C^{h,n}(\omega)$ is of the form
\[
\widehat C^{h,n}(\omega) = \alpha\omega^2 + \beta\omega +1,
\]
with
\begin{align*}
\alpha :=&\, 1 + \left(\frac{C_0(\tau)}{c}\right)\left(\frac{C_0(\tau)}{c} + C_1(h,\tau)\right) 
 +2\left(\frac{cC_0(\tau)}{(C_{PS}\sigma)^2}\left(1-C_1(h,\tau)\right)-C_1(h,\tau)\right),\\
\beta :=&\, -2\left(1-C_1(h,\tau)\right)\left(\frac{cC_0(\tau)}{(C_{PS}\sigma)^2} + 1\right).
\end{align*}

The extreme value for $\widehat C^{h,n}(\omega)$ is attained at
\[
\omega = \omega_m : = -\frac{\beta}{2\alpha}. 
\]
Since $C_1(h,\tau)$ vanishes as $h\to 0$, for a fine enough mesh it will hold that $\alpha>0$ and $\beta<0$. Therefore, $\omega_m$  will belong to the interval $(0,1)$ and will in fact be a minimizer of $\widehat C^{h,n}$. Moreover, since $\widehat C^{h,n}(0)=1$ and $\widehat C^{h,n}$ is decreasing in $(0,\omega_m)\subset(0,1)$, we conclude that it is possible to choose $\omega\in(0,1)$ such that $\Tomega^{h,n}$ is contractive. For these values of $\omega$, the convergence of the iterative process \eqref{eq:RelaxedIterativeProcess} follows from Banach's fixed-point theorem.
\end{proof}

We note that for the case of a fitted geometry (i.e. whenever $\Omega\equiv\Omega_h$) the distance parameter $R_h=0$. This implies that $C_1(h,\tau)=0$ and then 
\[
\widehat C^{h,n}(\omega)  = \left(\left(\frac{1+\overline\tau}{c}\right)^2 +\frac{2c(1+\overline\tau)}{(C_{PS}\sigma)^2} +1\right)\omega^2 -2 \left(1+\frac{c(1+\overline\tau)}{(C_{PS}\sigma)^2}\right)\omega +1,
\]
in coincidence with the continuous case. Above, the presence of the parameter $\overline\tau$ stems from the discretization, while the absence of factors involving $\boldsymbol\kappa$ is due to the choice of discrete norms.

\color{black}

\section*{Acknowledgments}

The authors have no relevant financial or non-financial interests to disclose. All authors have contributed equally to the article and the order of authorship has been determined alphabetically. Tonatiuh S\'anchez--Vizuet was partially supported by the National Science Foundation throught the grant NSF-DMS-2137305 ``LEAPS-MPS: Hybridizable discontinuous Galerkin methods for non-linear integro-differential boundary value problems in magnetic plasma confinement''.
Manuel Solano was supported by ANID--Chile through Fondecyt 1200569 and by Centro de Modelamiento Matemático (CMM), ACE210010 and FB210005, BASAL funds for center of excellence from ANID-Chile. \\

\noindent\rule{7cm}{1pt}\\
\noindent \textbf{Contact information}\\
Nestor S\'anchez: {\tt nestor\_sanchez@im.unam.mx}\\
Tonatiuh S\'anchez--Vizuet: {\tt tonatiuh@math.arizona.edu}\\
Manuel E. Solano: {\tt msolano@ing-mat.udec.cl}

\appendix
\setcounter{lem}{0}
\renewcommand{\thelem}{\Alph{section}\arabic{lem}}

\section{HDG projection.}\label{app:HDG-proj}

 Given constants $l_u, l_{\mbf{q}} \in [0,k]$,  $T\in \mc{T}_h$ and  a pair of functions $(\boldsymbol q,u) \in H^{1+l_q}(T) \times H^{1+l_u}(T)$,
  by \cite{CoGoSa2010} there is a constant $C>0$ independent of $T$ and $\tau$ such that   
    \begin{subequations}\label{C0:eq:projection_error}
    \begin{align}
	\|\bsy{\Pi}_{\mathrm v}\mbf{q} - \mbf{q}\|_T &\lesssim  h_T^{l_{\mbf{q}}+1} |\mbf{q}|_{l_{\mbf{q}}+1,T} +   h_T^{l_u+1} \tau_T^* |u|_{l_u+1,T}, \label{C0:error_projector1}\\
	\|\Pi_{\mathrm w} u - u\|_T &\lesssim  h_T^{l_u+1} |u|_{l_u+1,T} +  \dfrac{h_T^{l_{\mbf{q}}+1}}{\tau_T^{\max}} |\nabla \cdot \mbf{q}|_{l_{\mbf{q}}T}\label{C0:error_projector2},
	\end{align}
    \end{subequations}
where $\tau_T^* := \max \tau|_{\partial T \setminus F^*}$ and $F^*$ is a face of $T$ at which $\tau|_{\partial T}$ is maximum. As is customary, the symbol $|\cdot|_{H^s}$ is to be understood as the Sobolev semi norm of order $s\in\mathbb R$. 
Now, in the context of the unfitted HDG method, the projection errors  in $\Omega_h^c$ satisfies (Lemma 3.8 \cite{CoQiuSo2014})
\begin{subequations}
\begin{align*}
    	\|\bsy{\Pi}_{\mathrm v}\mbf{q} - \mbf{q}\|_{\Omega_h^c} &\lesssim  R_h^{1/2} 	\|\bsy{\Pi}_{\mathrm v}\mbf{q} - \mbf{q}\|_{\Omega_h}
    	+ h^{l_{\mbf{q}}+1} |\mbf{q}|_{l_{\mbf{q}},\Omega_h}.
\end{align*}
\end{subequations}
	
\bibliography{biblio}
\bibliographystyle{abbrv}
\include{biblio}


\end{document}